\documentclass[11pt,reqno]{amsart}

\usepackage{amsmath,amsfonts}
\usepackage{amssymb}
\usepackage{amscd}
\usepackage{amsthm}
\usepackage{yhmath}
\usepackage{booktabs}
\usepackage{subfigure}
\usepackage{graphicx}
\usepackage{blkarray}
\usepackage{xfrac}
\usepackage{}
\usepackage{comment}
\usepackage{stackrel}
\usepackage{tikz}
\usepackage[normalem]{ulem}
\usepackage{soul}

\usepackage{nicematrix}
\usepackage{mhchem}

\usepackage[all]{xy}
\usepackage{hyperref}

\usepackage{todonotes}

\numberwithin{equation}{section}

\newtheorem{theorem}{Theorem}[section]

\newtheorem{proposition}[theorem]{Proposition}

 \newtheorem{definition}[theorem]{Definition}
\newtheorem{lemma}[theorem]{Lemma}
\newtheorem{corollary}[theorem]{Corollary}
\newtheorem{conjecture}[theorem]{Conjecture}

\theoremstyle{remark}
\newtheorem{remark}[theorem]{Remark}

\newtheorem{example}[theorem]{Example}

\renewcommand{\phi}{\varphi}

\def\XXint#1#2#3{{\setbox0=\hbox{$#1{#2#3}{\int}$}
	\vcenter{\hbox{$#2#3$}}\kern-.5\wd0}}

  \renewcommand{\a}{\alpha}
  \renewcommand{\b}{\beta}

\newcommand{\N}{\mathbb N}
\newcommand{\Q}{\mathbb Q}
\newcommand{\Z}{\mathbb Z}
 \newcommand{\C}{\mathbb C}
 \newcommand{\R}{\mathbb R}
 \newcommand{\F}{\mathbb F}
\newcommand{\G}{{\mathbb G}}
 
 \newcommand{\wh}{\widehat}
\newcommand{\ol}{\overline}
\newcommand{\cut}{\textup{cut}}
\newcommand{\p}{\partial}
\renewcommand{\ker}{\operatorname{Ker}}
\newcommand{\length}{\operatorname{length}}
\newcommand{\Lie}{\operatorname{Lie}}

\newcommand{\rank}{\operatorname{rank}}
\newcommand{\Ker}{\operatorname{Ker}}
\newcommand{\Span}{\operatorname{span}}

\renewcommand{\H}{\mathbb{H}}
\newcommand{\Cut}{\operatorname{Cut}}
\renewcommand{\so}{\mathfrak{so}}

\newcommand{\s}{\sigma}
\newcommand{\la}{\lambda}
\renewcommand{\t}{\tau}

\begin{document}

\title[Length-extremals in rank-4 Carnot groups  ]{New properties of length-extremals in free  step-2  rank-4 Carnot groups
}
\author{Annamaria Montanari}
\address[Annamaria Montanari]{Dipartimento di Matematica, Alma Mater Studiorum Universit\`a di Bologna, Italy. \textit{Email: } \tt{annamaria.montanari@unibo.it}}

\author{Daniele Morbidelli   }
\address[Daniele Morbidelli (corresponding author)]{Dipartimento di Matematica, Alma Mater Studiorum Universit\`a di Bologna, Italy. \textit{Email: } \tt{daniele.morbidelli@unibo.it} }



\date{today}
\begin{abstract}
In the free, step-2, rank-4 sub-Riemannian Carnot group $\F_4$, we give a clean expression for length-extremals, we provide an explicit equation for conjugate points, we relate it with the conjectured cut locus   of the origin    $\Cut(\F_4)$. Finally, we give some upper estimates for the cut-time of extremals.
\end{abstract}

\date{\today}
\maketitle
\setcounter{tocdepth}{1}
\tableofcontents

\section{Introduction and main results}

We consider the free step-2, rank-$n$ Carnot group $\F_n =( \R^n\times\Lambda^2\R^n, \cdot) $ where the group law $\cdot$ is defined as
\begin{equation}\label{unouno}
  (x,t)\cdot(\xi,\t )=\Big( x+\xi, t+\tau +\frac 12 x\wedge\xi\Big)
 \end{equation}
for all $(x,t)$ and $(\xi,\t)\in\R^n\times\Lambda^2\R^n$. We equip $\R^n$  with the Euclidean inner product and we discusses some properties of the related  sub-Riemannian length-minimizing curves from $(0,0)$ (see Section~\ref{preliminari}, for precise definitions).  This topic has been discussed in~\cite{Brockett}, \cite{Myasnichenko02}, \cite{MonroyMeneses06}, \cite{RizziSerres16} and \cite{MontanariMorbidelli17}.   It is clear from the mentioned papers that,   in spite of the simple, dimension-free aspect
of~\eqref{unouno}, difficulties of doing analysis in $\F_n$ increase drastically with the \emph{rank}~$n\in\N$.
Before starting a specific description of the paper, let us mention that,   besides the free, step-2 model,     analysis   of length-minimizing properties of curves in   general Carnot groups and    sub-Riemannian  manifolds is a widely studied topic in modern geometric control theory.   See for instance~\cite{AgrachevBarilariBoscain} or~\cite{Sachkov22}, and see below for further references.

In this paper, among free, step-2  Carnot groups, we focus on the rank-4  case $\F_4$. This is a ten-dimensional model, since $\dim( \Lambda^2\R^4)=6$.
A careful study of previous contributions in~$\F_3$, \cite{Myasnichenko02,MontanariMorbidelli17,LiZhang},
shows that techniques in rank-3  case are
not easy to generalize in rank-4 or greater.  For instance, a generic element $t\in\Lambda^2\R^3$ is always decomposable, i.e.~it has the form $u\wedge v$ for suitable  $u,v\in\R^3$. In rank-4 this is a rare circumstance.
   Another related motivation can be seen if we identify $t\in \Lambda^2\R^n $ with the   skew-symmetric matrix $t\in \so(n)$, see below.
     In such case, the exponential of $t$ has a closed form only in $\Lambda^2\R^3$,
   becoming more difficult in $\Lambda^2 \R^n$ if $n\geq 4$.
Here,     to tackle the rank-4 case    we exploit a convenient way to write extremal curves, we write an explicit equation for conjugate points and we analyze some of its properties. We show that this equation factorizes into two factors. One of them essentially detects conjugate points which are also cut points belonging to the set  conjectured by Rizzi and Serres~\cite{RizziSerres16} as a candidate cut locus.
The other captures  a potentially huge set of conjugate points which are not expected  to be cut points, if the   aforementioned   conjecture would be confirmed. Using the explicit form of the equation for conjugate points,  we show that any non-rectilinear length-extremal curve  in $\F_4$ meets   the Rizzi-Serres set     infinitely many times.         For a  subclass of extremals, when suitable ``angular parameters'' are rationally dependent, we also obtain an upper estimate of the cut time.

The main object of our analysis are length-extremal curves from the origin. It is known that  such curves  $\gamma(s)= (x(s), t(s))$ can be obtained by integration of the ODE
\begin{equation}\label{fuss}
 \dot x= u\quad \text{ and }\quad \dot t =\frac 12 x\wedge u
\end{equation}
with initial condition  $\gamma(0)=(x(0), t(0))=(0,0)$ and taking $u:\R\to \R^4$ of the form
\begin{equation}\label{messico}
 u(s) = a_1\cos(2\phi_1 s)+ b_1\sin(2\phi_1 s)+ a_2\cos(2\phi_2s) + b_2\sin(2\phi_2 s),
\end{equation}
where $a_1, a_2, b_1, b_2\in \R^4$  are pairwise orthogonal,   $\phi_1,\phi_2\in\R$,   $|a_1|=|b_1|=:r_1\geq 0 $, $|a_2 |=|b_2|=:r_2\geq  0 $,       and without loss of generality $0\leq \phi_2\leq \phi_1 $.
These curves will be described better in Section~\ref{preliminari}.
  We observe already here that, given $a_1, b_1, a_2$ and $b_2$,  the control curve  $u(s)$ can be seen as a linear flow on a bidimensional torus. Some of the    proofs    later will depend on rationality/irrationality of the flow.   Integrating the control~\eqref{messico}, we get a curve $s\mapsto \gamma(s,a,b,\phi) =(x(s,a,b,\phi),t(s,a,b,\phi))$ of constant \emph{sub-Riemannian speed} $|\dot\gamma(s)|_{\rm{SR}}=:|u(s)|=\sqrt{r_1^2+r_2^2}>0$, whose \emph{sub-Riemannian length} is
\begin{equation*}
  \length(\gamma|_{[0,T]}):=\int_0^T |u(s)|ds = T \sqrt{r_1^2+ r_2^2}.
\end{equation*}
It turns out that for any $T>0$ sufficiently close to $0$, the curve $\gamma$ is a length-minimizer  among all
horizontal curves connecting $\gamma(0)=(0,0)$ and $\gamma(T)$. See Section~\ref{preliminari}.
The \emph{cut time} of the extremal curve $\gamma = \gamma(\cdot, a,b,\phi):\left[0,+\infty\right[\to \F_4$
is defined as follows
\begin{equation*}
 t_\cut(\gamma) =\sup\{T>0: \gamma|_{[0,T]}\text{ is a length-minimizer between $\gamma(0) $ and $\gamma(T)$\}.}
\end{equation*}
In general one can  have $ t_\cut(\gamma) \in\left]0,+\infty\right]$, depending on~$\gamma$.  The cut locus $\Cut(\F_4)\subset\F_4\setminus\{0\}$ is the set of all cut-points $\gamma(t_\cut)$ as $\gamma$ is a length-extremal.
Finding the cut-time of any given extremal and  detecting the cut locus   of a point   is a classical, sometimes difficult problem in sub-Riemannian geometry (at the end of the introduction we will give  some references).

Starting from the form~\eqref{messico} of extremals,~\cite{MonroyMeneses06} calculated  that, given $a_1, b_1, a_2, b_2$ and $ \phi_1, \phi_2$ as above,  the corresponding curve
$\gamma(s, a , b ,  \phi ) =(x(s, a , b ,  \phi ),t(s, a , b  , \phi ))$ has the form
\begin{equation}\label{essenza2}
 \begin{aligned}
x(s, a , b ,  \phi )&
 = sT(\phi_1 s)\big(a_1\cos(\phi_1 s)+b_1\sin(\phi_1 s)\big)
 +
 sT(\phi_2 s)\big(a_2\cos(\phi_2 s)+b_2\sin(\phi_2 s)\big)
\\
t(s,a,b,\phi)& =s^2 U(s\phi_1)a_1\wedge b_1 + s^2 F(s\phi_1, s\phi_2) a_1\wedge a_2
 \\&\quad +s^2 G(s\phi_1, s\phi_2)a_1\wedge b_2
 - s^2 G(s\phi_2, s\phi_1) b_1\wedge a_2
 \\ &\quad +s^2H(s\phi_1, s\phi_2)b_1\wedge b_2+ s^2 U(s\phi_2) a_2\wedge b_2.
 \end{aligned}\end{equation}
Here we defined $T(\phi)=\frac{\sin\phi}{\phi}$,  $U(\phi)=\frac{\phi-\sin\phi\cos\phi}{4\phi^2}$, while   the functions $F,G$ and $H$ are discussed in Section~\ref{preliminari}.    Although we use different notation,  formula~\eqref{essenza2} is analogous to~\cite[Theorem~6.1]{MonroyMeneses06} specialized to $n=4$.

In our first result,  we exploit a change of basis which makes the form of~\eqref{essenza2} much more manageable. This will enable us in the subsequent part to write an explicit equation for conjugate points.
To state our result, introduce for $0<\phi_2<\phi_1$ the function
 \begin{equation}\label{zeta}
    Z(\phi_1,\phi_2)=\frac{\phi_2\cos\phi_2\sin\phi_1-\phi_1\cos\phi_1\sin\phi_2}{2\phi_2(\phi_1^2-\phi_2^2)}.
 \end{equation}

\begin{theorem} \label{sempres} Let $a_1, b_1, a_2, b_2$ be  pairwise orthogonal with $|a_k|=|b_k|$ for $k=1,2$, and let $\phi_1>\phi_2>0$ be given. Consider the extremal $\gamma(s,a,b, \phi)$ in~\eqref{essenza2}.
Let then
\begin{equation}\label{alfacappa}
\left\{ \begin{aligned}
 & \a_k= a_k\sin\phi_k -b_k\cos\phi_k=: a_k s_k -b_k c_k \qquad
 \\&
 \beta_k = a_k\cos  \phi_k  +b_k\sin \phi_k   =: a_k c_k + b_k s_k                                                                  \end{aligned}\quad  \text{for $k=1,2$.}
\right.
\end{equation}
Then we have  \begin{equation}\label{super}
\left\{\begin{aligned}
  x(1, a,b , \phi)& = T(\phi_1)\b_1+T(\phi_2)\b_2
  =: T_1 \b_1 + T_2 \b_2
 \\
 t(1,a,b ,\phi)&= U(\phi_1)\a_1\wedge\beta_1+Z(\phi_1,\phi_2)\a_1\wedge\b_2
 \\& \qquad\qquad +Z(\phi_2,\phi_1)\a_2\wedge \b_1 +U(\phi_2)\a_2\wedge\b_2
 \\& =: U_1 \a_1\wedge\beta_1+Z_{12} \a_1\wedge\b_2
 + Z_{21} \a_2\wedge\b_1+U_2 \a_2\wedge\b_2.
\end{aligned}\right.
 \end{equation}
\end{theorem}
Observe the abridged notation $T_k:=T(\phi_k)$,  $U_k:=U(\phi_k) $, $Z_{jk}:= Z(\phi_j, \phi_k) $, $c_k=\cos\phi_k$ and $s_k=\sin\phi_k$ for $k=1, 2$.   This notation will be used  frequently below. This theorem gives the form of the extremal curve at time $s=1$. However, the reparametrization property
  $\gamma(s,a,b,\phi)=\gamma(1, sa, sb, s\phi)$ for all $s>0$ and for all   $a,b,\phi$ gives the form of $\gamma$ in terms of the functions $U $ and $Z $   for all times $s$, see Corollary~\ref{penna}.    Note that   we state Theorem~\ref{sempres} taking strict inequalities $\phi_1\gneqq\phi_2\gneqq 0$, $r_1=|a_1|$ and $ r_2=|a_2|>0$. All degenerate cases will be included in Subsection~\ref{pino}.
  Observe that the term   $t(1,a,b,\phi)$   in~\eqref{super} has only four nonzero terms instead of six, as it was in~\eqref{essenza2}.

 From now on, we will always identify $ \Lambda^2\R^4$ with $  \so (4)$, the vector space of skew-symmetric matrices,  by  extending linearly the identification     $u\wedge v \simeq uv^T-vu^T\in \so(4)$   for all $u,v\in\R^4$, see Section~\ref{preliminari}. Under this identification, it turns out that in the ordered orthonormal basis $ u_1:=\frac{\a_1}{r_1},u_2: =\frac{\a_2}{r_2}$ , $ v_1:=\frac{\b_1}{r_1}$ and $ v_2:=\frac{\b_2}{r_2}$, where $r_k=|\a_k|=|\b_k|$ for $k=1,2$,  the matrix $ t(1,a,b,\phi)\in \so(4)$ appearing in~\eqref{super}, has the block form
\begin{equation}\label{blocco}
   t(1,a,b,\phi)=\begin{bmatrix}
           0  & M \\ -M^T & 0
          \end{bmatrix}\in\so(4), \quad \text{where   }\quad
M=\begin{bmatrix}
r_1^2 U_1 & r_1 r_2 Z_{12}\\ r_1 r_2 Z_{21} & r_2^2 U_2          \end{bmatrix}\in\R^{2\times 2}.
\end{equation}
This makes several computations simpler. However, it must be observed that, in spite of the block-form, eigenvalues and eigenvectors of antisymmetric matrices of the form~\eqref{blocco} are quite complicated to express in terms of the variables $r_k $ and $\phi_k$. See a partial discussion in Remark~\ref{conforme}.

 Starting from the previous result,  we come to the main part of the paper, where
we    analyze whether or not the point $\gamma(1)$ in~\eqref{super} is conjugate to $\gamma(0)=(0,0)$ along $\gamma$. Following~\cite{AgrachevBarilariBoscain}, in order to analyze such condition, we should write $\gamma(s,a,b,\phi)=\exp(s(\xi,\t))$, where $(\xi,\t )\in T^*_{(0,0)}(\R^4\times\Lambda^2\R^4)$ and $\exp:T^*_{(0,0)}(\R^4\times\Lambda^2\R^4)\to  \R^4\times\Lambda^2\R^4  $ denotes the sub-Riemannian exponential.
Then, by definition,  the point $\gamma(1)=\exp(\xi,\t)\in\F_4$ is conjugate if the differential $d_{(\xi,\t)}\exp$ is singular.  However,  in the present paper  we do not use standard Hamiltonian coordinates $(\xi,\t)$, but we take coordinates modeled on the parameters $\a,\b,\phi$  appearing in~\eqref{alfacappa}. This choice will capture automatically  the orthogonal invariance of the problem, which will be described in Subsection~\ref{scorrevole}.

To state our results, start from the extremal $\gamma(\cdot , a,b,\phi)=(x(\cdot , a,b,\phi),t(\cdot , a,b,\phi))$ in~\eqref{essenza2}. Introduce $\a_k$ and $\b_k$ in terms of $a_k,b_k,\phi_k$ by the rotation in~\eqref{alfacappa}.  Let $u_k:=\frac{\a_k}{|\a_k|}=\frac{\a_k}{r_k}  $ and $  v_k=\frac{\b_k}{r_k}$ for $k=1,2$. Then, the point $\gamma(1 , a,b,\phi)$
is uniquely determined by the parameters
\begin{equation*}
(u,v,r,\phi):=( (u_1, v_1, u_2,   v_2), (r_1, r_2, \phi_1, \phi_2))\in\Sigma\times\Omega,
\end{equation*}
where $\Sigma = \{(x_1, x_2, x_3, x_4): x_1, x_2, x_3, x_4\text{ are orthonormal in $\R^4$\}}\subset\R^{16}$ and $\Omega=\{(r_1, r_2, $ $ \phi_1, \phi_2)\in\left]0,+\infty\right[^4$ such that $\phi_1>\phi_2>0\}$.
We may denote then
\begin{equation}\label{bbb}
 \Gamma(u,v,r,\phi):  = \gamma(1, a,b,\phi).
\end{equation}
It turns out that the  ten dimensional manifold $\Sigma\times\Omega$ is diffeomorphic  to the following set of ``nondegenerate'' covectors
$G:=\{(\xi,\t)\in\R^4\times\Lambda^2\R^4: \xi\neq 0$   and $\t$   has four distinct nonzero eigenvalues $\} \subset T^*_{(0,0)}\F_4$, see Proposition~\ref{coperta}.
We will see that  the point in~\eqref{bbb}  is conjugate  to the origin along $\gamma(\cdot, a,b,\phi )$ at time $s=1$ if and only if $d_{(u,v,r,\phi)}\Gamma$ is singular.
A careful calculation of the differential of the map~$\Gamma:\Sigma\times\Omega\to \F_4$ gives then  the following theorem
\begin{theorem}\label{firenze}
 Let $(u,v,r,\phi)=( (u_1, v_1, u_2, v_2),( r_1, r_2, \phi_1, \phi_2))\in\Sigma\times\Omega$. Then, the point
 \begin{equation}\label{arezzo}
 \begin{aligned}
& \gamma(1, a,b,\phi)= \Gamma(u,v,r,\phi)
 \\&  =(r_1 T_1 v_1+ r_2 T_2v_2, r_1^2 U_1 u_1\wedge v_1 + r_1 r_2 Z_{12} u_1\wedge v_2
 + r_1 r_2Z_{21}u_2\wedge v_1 + r_2^2 U_2 u_2\wedge v_2 )
\end{aligned}
 \end{equation}
 is conjugate to $(0,0)$ along $s\mapsto \gamma(s,a,b,\phi)$ if and only if  at least one of the following two square matrices is singular:
\begin{equation}
\label{M11}
M_1 =
\left[\begin{smallmatrix}
   -T_1 & 0 & -r_2^2 T_2 & 0
\\ \\
   0 & - T_2 & 0 & -r_1^2T_1
\\ \\
Z_{21} & - Z_{12} &      r_2^2 U_2  &  -r_1^2 U_1
\\ \\
 Z_{12}&-Z_{21}  & - r_1^2   U_1  &
r_2^2 U_2
\end{smallmatrix}\right]
\end{equation}
or
\begin{equation}\label{trentasei}
 M_2=\left[\begin{smallmatrix}
 0 & -r_2^2 T_2 & T_1 & 0 & -2V_1 & 0
\\ \\
    0 & r_1^2 T_1 & 0 & T_2 & 0 & -2V_2
\\ \\
     -r_2^2 Z_{21}& -r_2^2Z_{12} & 2U_1 & 0  &  \frac{\cos \phi_1}{\phi_1} V_1 & 0
 \\ \\
     -r_2^2 U_2 & r_1^2 U_1 & Z_{12} & Z_{12} &  (\p_1 Z)(\phi_1, \phi_2) &    (\p_2 Z)(\phi_1, \phi_2)
\\ \\
     -r_1^2U_1 & r_2^2 U_2 & -Z_{21} & -Z_{21} & -( \p_2 Z)(\phi_2, \phi_1) & - (\p_1Z)(\phi_2, \phi_1)
\\ \\
    r_1^2 Z_{12} & r_1^2 Z_{21} & 0 & 2U_2 & 0 &\frac{\cos\phi_2}{\phi_2}V_2
\end{smallmatrix}\right].
\end{equation}
\end{theorem}
Note that $M_1$ and $M_2$ do not depend on $u_1, v_1, u_2, v_2$, by the orthogonal invariance of the model, see Section~\ref{preliminari}. This factorization property while calculating conjugate points is not unexpected. Indeed it already appears in  Myasnichenko's paper in rank-3  case, see~\cite[Eq.~(12), p.~586]{Myasnichenko02}. In rank-3  case only one of the  factors gives conjugate points which belong to the cut locus. In our rank-4 case this is actually an open question, see below.
As expected, Theorem~\ref{firenze}   does   not give  information on whether the point $\Gamma(u,v,r,\phi)$ is the \emph{first conjugate point}. \footnote{Recall that a point $\gamma(\bar s)$ on an extremal $\gamma$ is the first conjugate point if there are no other conjugate points in $\left]0,\bar s  \right[$. See~\cite[Definition~8.45]{AgrachevBarilariBoscain} for the precise definition.}

In the   subsequent part of the paper,   we discuss condition $\det M_1=0$ and we analyze how this condition relates with the fact that $\Gamma(u,v,r,\phi)$ belongs to $C_4$, the candidate cut locus  proposed  by Rizzi and Serres~\cite{RizziSerres16}, see~\eqref{ciquattro}.
In order to state our result, introduce the functions
\begin{equation}
  A(\phi_1, \phi_2)=T_1 U_1 \big(T_1Z_{12}-T_2 U_1\big) \quad\text{and}\quad  B(\phi_1, \phi_2)= T_2 Z_{12} \big(T_1Z_{12} - T_2 U_1\big),
\end{equation}
where as usual $T_k=T(\phi_k)$ and $Z_{ij}=Z(\phi_i,\phi_j)$ for $i,j,k=1,2$.
  The following theorem extracts
 some useful information   concerning points where $  M_1$ is singular.
\begin{theorem} \label{MMM}  Let $(u_1,v_1, u_2, v_2)\in\Sigma$ and consider $r_1, r_2>0$ and $0<\phi_2< \phi_1$. Take the point  $(x,t):=\Gamma(u,v,r,\phi)$ appearing in~\eqref{arezzo}. Identifying as usual $ \Lambda^2\R^4$ and~$  \mathfrak{so}(4)$,
the following statements are equivalent:
\begin{enumerate}
 \item \label{primavera2}  $\det M_1=0$.
 \item \label{autunno2}
 $  t{\,}^2 x \in\Span\{x\}$.
\item The following equation holds
\begin{equation}\label{AB}
 A(\phi_1, \phi_2)r_1^4  + \{ B(\phi_1, \phi_2)- B(\phi_2, \phi_1) \}
  r_1^2 r_2^2  -A(\phi_2, \phi_1)r_2^2=0.
\end{equation}
 \end{enumerate}
\end{theorem}

As expected, equation~\eqref{AB} is invariant with respect to exchanging of indices $1$ and $2$. It degenerates correctly to the known formulas from \cite{MontanariMorbidelli17} in $\F_3$, as $\phi_2\to 0$. Namely, it becomes
 $\frac{r_2^2}{r_1^2}=-\frac{T(\phi_1)  U(\phi_1)}{V(\phi_1)}$, where $T$ and $U $ appeared above, while $V(\phi_1)=Z(\phi_1, 0)$, see Remark~\ref{aosta} and~\eqref{azzero}.

Condition (\ref{autunno2}) is  rather interesting, because it relates with the Rizzi-Serres conjectured set~$C_4$.
Recall that  in the paper \cite{RizziSerres16}, Rizzi and Serres conjectured that the cut locus $\Cut(\F_4): =\{\gamma(t_\cut ):$ $\gamma$ is a length-extremal  and $t_\cut<\infty\}$ agrees with the  set $C_4 =  \Sigma_1\cup\Sigma_2\cup\Sigma_3$,
 where,
 \begin{equation}\label{ciquattro}
\begin{aligned}
  \Sigma_1 &=\{(x_1 u_1 + x_2 u_2, \la_1 u_1\wedge u_2+\la_2 u_3 \wedge u_4):\la_1,\la_2>0\quad   \la_2\neq\la_1
 \; \text{and }(x_1, x_2)\in\R^2\},
\\  \Sigma_2& =\big\{ (x, \la (u_1\wedge u_2 + u_3\wedge u_4)):\la>0 \text{ and $ x\in\R^4$}\big\},
\\ \Sigma_3 &=\big\{(x_1 u_1, \la u_3\wedge u_4): \la>0,\quad x_1\in\R \big\}.
\end{aligned}\end{equation}
 In the previous formula,  $ u_1, u_2, u_3, u_4 $  denote    any orthonormal family in $\R^4$.
 By orthogonal invariance,  it is rather easy to see that $C_4\subset \Cut(\F_4)$  (see~\cite[Proposition~6]{RizziSerres16}).  The opposite inclusion is an open problem. See Section~\ref{preliminari} for a more detailed discussion and for the dimension-free  definition of $C_n\subset\F_n$ formulated in \cite{RizziSerres16}.

Let us come  now to the aforementioned relation between condition~(\ref{autunno2}) of Theorem~\ref{MMM} and the set~$C_4$. It is not difficult to see that, if $(x,t)\in\F_4$ and $\rank(  t)= 4$, then $  t^2 x\in\Span\{x\}$ if and only if $(x,t)\in\Sigma_1\cup\Sigma_2$ (see Remark~\ref{competizione}).  In other words,  if $  t$ has rank-4, then any of the equivalent conditions in Theorem~\ref{MMM} is
 equivalent to $(x,t)\in C_4$. Note that if $\rank(   t) =2$, then the equivalence fails. For instance, the point $(x_1 e_1, e_1\wedge e_2 )$ satisfies $ t^2 x\in\Span\{x\}$, but if $x_1\neq 0$, then $(x,t)\notin C_4$, and the discussion of Subsection~\ref{pino} also shows that $(x,t)\notin\Cut(\F_4)$.
Analyzing equation~\eqref{AB}, it turns out however that  given $0<\phi_2<\phi_1$, there is $\ol s>0$ such that the curve  $\gamma(s, u,v,r,\phi)=:(x(s), t(s))$ satisfies $\rank(  t(s))=4$   for~$s>\ol s$ and for any $r,u,v$. See Lemma~\ref{deboluccio}.
We expect  that  $\ol s=0$, but the proof would be based on the achievement of a rather difficult inequality   discussed  in Remarks~\ref{messina} and~\ref{Taylor}.

It is well known that   in the Heisenberg group~$\F_2$
with coordinates $(x,y,t)\in\R^3$, any non-rectilinear length-extremal from the origin touches the $t$-axis infinitely many times, in a periodical way. The same happens in the rank-3  case, as shown in~\cite{MontanariMorbidelli17,Myasnichenko02}: the conjugate locus is touched infinitely many times,  but  periodicity no longer holds.
In the subsequent part of the paper, starting from equation~\eqref{AB}, we show an analogous phenomenon in $\F_4$. Observe that, if $\phi_1$ and $\phi_2$ are rationally dependent, then it is trivial to see that the set $C_4$ is reached infinitely many times. This follows from the fact that the function $s\mapsto x(s)$ in~\eqref{essenza2} is periodic, and there is $\bar s>0$ such that $x(k\bar s)=0\in\R^4$ for all $k\in\N\cup\{0\}$.
The rationally independent case requires more work.
In view of the  greater technical difficulty, we get the result for large times only, using the behaviour at infinity of equation~\eqref{AB}. The theorem below   is meaningful for \emph{strictly normal} curves, which are those such that $r_1$ and $r_2$ are both strictly positive, $\phi_1\gneqq   \phi_2\geq 0$,   see Subsection~\ref{duetre}.
 Concerning abnormal extremals,  by second-order analysis of the end point map, in~$\F_4$---as in all step-two Carnot groups---abnormal extremals are also normal,   see~\cite[Theorem~12.12 and Corollary~12.14]{AgrachevBarilariBoscain}.
Therefore, in our case, points of an  abnormal length-extremal are     conjugate by a standard easy argument. As a final remark, note that  the case $\phi_2=0$ is also already known, being  contained in~\cite{Myasnichenko02,MontanariMorbidelli17}.

\begin{theorem}\label{sarebbero}
 Let $u(s)= \sum_{k=1}^2  a_k\cos(2\phi_k s) + b_k\sin(2\phi_k s)$ be a  strictly normal   control.  Consider the corresponding trajectory $\gamma(\cdot,a ,b , \phi )$. Then there is a sequence $s_j\to +\infty$ such that  $\gamma(s_j,a,b,\phi) \in   C_4$ for all $j\in\N$.
\end{theorem}
As we already said, the proof  is easy if $\phi_1$ and $\phi_2$ are rationally dependent. The even more particular case $\phi_1=2\phi_2>0$ appears in Brockett's paper \cite{Brockett}. In Section~\ref{prima}, we prove the general rationally independent case.
Although Theorem~\ref{sarebbero} does not give information about the cut-time, it   implies immediately that the cut time is finite for any non-rectilinear length-extremal. This confirms a result by Kishimoto \cite{Kishimoto}.
\footnote{  Note that analogous finiteness results fail in Carnot groups of step three and higher. Namely, there are geodesics with $t_{\cut}=+\infty$ and which are not integral curves of left-invariant horizontal vector fields. See \cite{ArdentovSachkov14}, \cite[Section 7.2]{HakavuoriLeDonne}, \cite{BravoDoddoli}.}

Before closing the introduction, we mention some further references on the problem
 of  the cut locus in  nonfree  Carnot groups. In the step-2 case, we mention the papers~\cite{MBarilariBoscainNeel16}, \cite{BarilariBoscainGauthier12}, \cite{AutenriedMolina16} and \cite{MontanariMorbidelli24}. See also~\cite{Li} for a different approach.
 In step-3 we mention~\cite{ArdentovSachkov14} on the Engel group. All these references and many others, also outside
 the setting of Carnot groups, are discussed in the comprehensive survey~\cite{Sachkov22}.

The structure of the paper is the following: in Section~\ref{preliminari} we write extremal curves, we analyze the change of basis useful to simplify them. In Section~\ref{Piperno} we  find  the equation for conjugate points. In Section~\ref{quartina} we analyze  conjugate points coming from the first factor $\det M_1=0$, those belonging to the Rizzi-Serres set. Section~\ref{prima} is devoted to the proof of Theorem~\ref{sarebbero}.

\section{General preliminaries  and extremal curves}  \label{preliminari}
 Let us consider in  $  \R^4\times\Lambda^2\R^4$ the Lie group law
 \begin{equation*}
  (x,t)\cdot(\xi,\t )=\Big( x+\xi, t+\tau +\frac 12 x\wedge\xi\Big).
 \end{equation*}
 It turns out that $\F_4=(\R^4\times\Lambda^2\R^4,\cdot)$ is a model for the free step-2  Carnot group of rank~4.  See the monographs~\cite{BonfiglioliLanconelliUguzzoni,AgrachevBarilariBoscain}.  We say that a Lipschitz curve $\gamma=(x,t):[0,T]\to \F_4$ is horizontal if it satisfies almost everywhere the ODE
 \begin{equation}
 \label{equino}
 \dot x = u,\qquad \dot t=\frac 12 x\wedge u.
\end{equation}
To define a sub-Riemannian structure, we fix on $\R^4$ the standard Euclidean inner product.
Then, the length of a horizontal curve $\gamma$ on $[0,T]$ is defined as
$
 \length(\gamma):=\int_0^T |u(s)|ds.
$
Minimizing length we obtain the sub-Riemannian distance
$ d((x,t),(\xi,\t))$ $ =\inf\{\length(\gamma): \gamma$    connects $(x,t)$ and $(\xi,\t)\}$.
It is well known that $d((x,t),(\xi,\t))$ is finite and it is a minimum for all~$(x,t)$ and $(\xi,\t)\in\F_4$.

As already mentioned in the Introduction, we identify $\Lambda^2 \R^n\simeq \so(n)$  extending linearly the identification   $u\wedge v \simeq   uv^T-vu^T\in \so(n)$  for all $u,v\in\R^n$.

\subsection{Hamiltonian approach, extremal controls and conjugate points} In order to write length-minimizing curves, we follow the Hamiltonian approach, see~\cite[Chapter~13.1]{AgrachevBarilariBoscain}.
 Let $e_1, \dots, e_n$ denote the standard basis of $\R^n$.  Given the  orthonormal  frame of horizontal vector fields in $\F_n$,
$X_j(x,t) =\big(e_j,\frac 12 x\wedge e_j\big)$ for $j=1,\dots,n$, we construct the functions
 $u_j:T^* \F_n\to \R   $ letting  $u_j(x,t,\xi,\t):=\langle  (\xi,\t),X_j(x,t)\rangle  $.  Here on $\Lambda^2\R^n$ we take the standard inner product making $e_j\wedge e_k$ an orthonormal system, as $1\leq j<k\leq n$. We are also identifying $T^*\F_n\simeq \F_n\times\F_n$.
The related sub-Riemannian Hamiltonian has the form
\[H(x,t,\xi,\t)=\frac 12\sum_{k=1}^nu_k(x,t,\xi,\t)^2=\frac 12\sum_{k=1}^n\langle X_k(x,t),(\xi,\t)\rangle^2.\]
Integrating the Hamiltonian system
\begin{equation}\label{ham}
\begin{aligned}
 \left\{ \begin{aligned}
&(\dot x,\dot t)=\Big(\frac{\p H}{\p \xi} (x,t,\xi,\t), \frac{\p H}{\p \t} (x,t,\xi,\t)\Big)
\\& (\dot\xi, \dot\t)=-
\Big(\frac{\p H}{\p x} (x,t,\xi,\t), \frac{\p H}{\p t} (x,t,\xi,\t)\Big)
\end{aligned}\right.
 \text{ with  }
 \left\{\begin{aligned}
&(x(0), t(0))=(0,0)\\&(\xi(0), \t(0))=(\xi,\t),
\end{aligned}\right.
\end{aligned}
\end{equation}    we obtain a  length-extremal curve $\gamma(\cdot, \xi,\t)=(x(\cdot,\xi,\t), t(\cdot,\xi,\t))$ starting from the origin.  Since $\F_4$ has step two, there are no strictly abnormal length-minimizers (see~\cite[Corollary~12.5]{AgrachevBarilariBoscain}). Thus,  it turns out that all length-extremals from the origin have the form $\gamma(\cdot, \xi, \t)$  and are parametrized by their
 initial covector $(\xi,\t)\in T^*_{(0,0)}\F_n$. They are defined for all $s\in\R$ and they enjoy property $\gamma(s,\xi,\t)=\gamma(1,s\xi, s\t)$ for all $s\in\R$ and $(\xi,\t)\in T_{(0,0)}^*\F_n$. For any $ \xi$ and $\t$, the extremal $\gamma(
 \cdot, \xi,\t)$ is a length-minimizer on some nontrivial interval   $[0,T]  $.

 Following \cite[Section~8.6]{AgrachevBarilariBoscain}, we define then the \emph{sub-Riemannian exponential} $\exp: T^*_{(0,0)}\F_n\to \F_n$, as $\exp(\xi,\t):=\gamma(1,\xi,\t)$.
 \begin{definition}\label{labello}
Given $(\xi,\t)\in\F_4$, we say that the point $\gamma(\ol s, \xi, \t)=\exp(\bar s \xi,\bar s \t)$ is \emph{conjugate} to $(0,0)$
 along $\gamma(\cdot,
 \xi,  \t )$ if the differential of $\exp$ at point $(\bar s  \xi,\bar s \t )$ is singular, i.e.
 \begin{equation}
d_{(\bar s  \xi,\bar s \t )}\exp \text{ is singular.}\end{equation}
Given $(\xi,\t)\in\F_n$  and the corresponding curve $\gamma=\gamma(\cdot,\xi,\t)$, we define
$
 t_\cut(\gamma) =\sup\{T\geq 0: \gamma|_{[0,T]}$  minimizes length
among  all $\gamma$ connecting $\gamma(0)$ and $\gamma(T)\}$.
Finally, the \emph{cut locus of the origin} of $\F_n$ is
\begin{equation*}
 \cut(\F_n):= \{\gamma( t_\cut): \text{$\gamma$ is an extremal and $t_\cut(\gamma)<\infty$}\}.
\end{equation*}
\end{definition}
Concerning the definition above, it is known that $t_\cut\in\left]0,+\infty\right]$.  Let us stress again that the inequality $t_{\cut}\gneqq 0 $ follows from the fact that there are no strictly abnormal extremals.

 Integration of the Hamiltonian system~\eqref{ham}  gives that  the extremal control $\gamma(\cdot, \xi, \t)$ is obtained by taking the  control
 \begin{equation}\label{letto}
  u(s,\xi,\t)= e^{-s \t}\xi
 \end{equation}
in the ODE~\eqref{equino} (See \cite[Section~13.3]{AgrachevBarilariBoscain}).
In this paper, we work on extremal controls of the  form~\eqref{letto}. Using spectral  theory of skew-symmetric matrices, it turns out  that, given a control of the form~\eqref{letto} in $\F_4$, we have
\begin{equation}\label{modellino}
 u(s)=a_1\cos(\la_1 s)+b_1\sin(\la_1 s)+a_2\cos(\la_2 s)+b_2\sin(\la_2 s)
\end{equation}
where $\la_1\geq \la_2\geq 0$, $r_k:=|a_k|=|b_k|\geq 0$ for $k=1,2$, and $a_1, a_2, b_1, b_2$ are pairwise orthogonal. See~\cite{AgrachevBarilariBoscain}, or see also the previous papers~\cite{Myasnichenko02,MonroyMeneses06,MontanariMorbidelli17,RizziSerres16}, where such extremal controls are already used.
\begin{definition}\label{generico}   We say that the extremal $u$ in~\eqref{modellino} is generic if $r_1, r_2 \gneqq 0$ and $\la_1\gneqq \la_2\gneqq 0$.
\end{definition}
A substantial part of our work will take place on generic extremals.

\subsection{Abnormal curves in \texorpdfstring{$\F_4$}{F4}}\label{duetre}
In order to talk about conjugate points, we need to discuss briefly abnormal extremal curves. For any given control  $u\in L^2((0,1), \R^4) $, define   the \emph{endpoint map} $E(u)=\gamma_u(1)$, where $\gamma_u=(x_u, t_u)$ is obtained by integration of~\eqref{equino}. It turns out that $E: L^2\to \F_4$ is a smooth map, see~\cite{AgrachevBarilariBoscain}. We say that a control $u$ is \emph{abnormal} if the differential
$d_u E: L^2\to \F_4$ is singular. By   well known theory  of Carnot groups,   an extremal control of the form~\eqref{modellino} is abnormal if and only if it takes the form
\begin{equation}\label{abno}
 u(s) =a_1\cos(\la_1 s)+b_1\sin(\la_1 s),
\end{equation}
where, as in~\eqref{modellino}, $a_1$ and $b_1$ are orthogonal and have the same norm and $\la_1\geq 0$. See for example~\cite{LeDonneLeonardiMontiVittone} and~\cite{MontanariMorbidelli17}.
Note that if $u(\cdot, \xi, \t)$ is an extremal abnormal control  (which is always  normal, because we work in  a step-2 group),  then for all $s> 0  $ the point $\exp(s\xi,s\t)$ is conjugate in the sense of Definition~\ref{labello}. This easy fact is observed in~\cite[Remark~8.46]{AgrachevBarilariBoscain}. Note that given the abnormal control~\eqref{abno}, the corresponding curve $\gamma_u$ is contained in the Heisenberg subgroup $\Lie(a_1, b_1)$ (see Subsection~\ref{pino}). Extremal controls of the form $u(s)= a_1\cos(\la_1 s)+ b_1\sin(\la_1 s)+ a_2$  with $\la_1>0$ and $a_1, b_1, a_2\neq 0 $ are instead not abnormal.

\subsection{Symmetries of \texorpdfstring{$\F_n$}{Fn} and the conjectured cut locus}\label{scorrevole}
Given $R\in O(n)$, we define the linear map  $\bar R:\F_n\to \F_n$ by $(x,t)\mapsto  (Rx,R t R^T )$. It is easy to see that the class  of horizontal curves and their length are invariant under the map  $\bar R$    for all $R\in O(n)$. Consequently, $d((x,t), (x', t'))= d(\bar R(x,t), \bar R(x', t'))$ for all pair of points and for all $R\in O(n)$.

It was conjectured by Rizzi and Serres~\cite{RizziSerres16} that $ \Cut(\F_n)=C_n$, where $C_n\subset\F_n$ is defined as follows.
\begin{equation}\label{jos}
\begin{aligned}
C_n =\{(x,t): & \text{ there is $R\in O(n)$, $R\neq I_n$ such that $\bar R(x,t) = (x,t)$}
 \\&\text{ and $R|_{\Ker t}=I|_{\Ker t}$}\}.
\end{aligned}
\end{equation}
For the proof that $C_n\subset \Cut(\F_n)$, see~\cite[Proposition~6]{RizziSerres16}. The equality $C_n =\Cut(\F_n)$
 is an open conjecture. For the case of our interest $n=4$, the set $C_4$ is described in~\eqref{ciquattro}.

\subsection{Extremal trajectories}
Let us go to the generic extremal control $u$ in~\eqref{modellino}. In order to integrate it, introduce the following functions:
 \begin{equation}\label{vecchie}
  T(\phi)=\frac{\sin\phi}{\phi}, \quad U(\phi)=\frac{\phi-\sin\phi\cos\phi}{4\phi^2},\quad
  V(\phi)=\frac{\sin\phi-\phi\cos\phi}{2\phi^2},
 \end{equation}
 \begin{equation*}
  F(\phi_1,\phi_2)=\frac 18\bigg\{\Big(\frac{1}{\phi_1}-\frac{1}{\phi_2}\Big)\frac{\sin^2(\phi_1+\phi_2)}{\phi_1+\phi_2}
  +
\Big(\frac{1}{\phi_1}+\frac{1}{\phi_2}\Big)
\frac{\sin^2(\phi_1-\phi_2)}{\phi_1-\phi_2}    \bigg\},
 \end{equation*}
 \begin{equation*}
 \begin{aligned}
 G(\phi_1,\phi_2)=\frac 18\bigg\{ & \Big(\frac{1}{\phi_1}+\frac{1}{\phi_2}\Big)
  \frac{\sin(\phi_1-\phi_2)\cos(\phi_1-\phi_2)}{\phi_1-\phi_2}
 \\&+
  \Big(\frac{1}{\phi_2 }-\frac{1}{\phi_1}\Big)
  \frac{\sin(\phi_1+\phi_2)\cos(\phi_1+\phi_2)}{\phi_1+\phi_2}
  -\frac{2\sin\phi_1\cos\phi_1}{\phi_1\phi_2}
  \bigg\}
  \end{aligned}
\end{equation*}
and
\begin{equation*}
\begin{aligned}
 H(\phi_1,\phi_2)=\frac 18\bigg\{ & \Big(\frac{1}{\phi_1}+\frac{1}{\phi_2}\Big) \frac{\sin^2(\phi_1-\phi_2)}{\phi_1-\phi_2}
 -
 \Big(\frac{1}{\phi_1}-\frac{1}{\phi_2}\Big) \frac{\sin^2(\phi_1+\phi_2)}{\phi_1+\phi_2}
 \\& +\frac{2(\sin^2\phi_2-\sin^2\phi_1)}{\phi_1 \phi_2}
 \bigg\}.
\end{aligned}
\end{equation*}

%

\begin{proposition}[Extremal trajectories] \label{lll}
 Let $  a_1, a_2, b_1, b_2$ be pairwise orthogonal and assume that $|a_k|=|b_k|= r_k>0$ for $k=1,2$. Consider for $\la_1>\la_2>0$ the corresponding generic extremal control
 \begin{equation}\label{ulisse}
\begin{aligned}
 u(s)& =a_1\cos(\la_1 s)+b_1\sin(\la_1 s)+a_2\cos(\la_2 s)+b_2\sin(\la_2 s)
 \\&
= :a_1\cos(2\phi_1 s)+b_1\sin(2\phi_1 s)+a_2\cos(2\phi_2 s)+b_2\sin(2\phi_2 s).
\end{aligned}
\end{equation}
 Then, the corresponding trajectory $\gamma(\cdot, a,b,\phi)=(x(\cdot, a,b,\phi),t(\cdot, a,b,\phi))$ has the form
\begin{equation}\label{essenza}
 \begin{aligned}
x(s) &
 = sT(\phi_1 s)\Big(a_1\cos(\phi_1 s)+b_1\sin(\phi_1 s)\Big) +
 sT(\phi_2 s)\Big(a_2\cos(\phi_2 s)+b_2\sin(\phi_2 s)\Big)
\\
t(s)& =s^2 U(s\phi_1)a_1\wedge b_1 + s^2 F(s\phi_1, s\phi_2) a_1\wedge a_2
 +s^2 G(s\phi_1, s\phi_2)a_1\wedge b_2
 \\ &\qquad - s^2 G(s\phi_2, s\phi_1) b_1\wedge a_2
 +s^2H(s\phi_1, s\phi_2)b_1\wedge b_2+ s^2 U(s\phi_2) a_2\wedge b_2.
 \end{aligned}\end{equation}
\end{proposition}
Concerning the extremal curves above, observe the reparametrization property
\begin{equation}\label{parametro}
 \gamma(s,a,b,\phi)= \gamma(1, sa, sb, s\phi) \quad \text{ for all $s>0$, $a,b,\phi$.}
\end{equation}
Proposition~\ref{lll} is proved in arbitrary dimension in~\cite{MonroyMeneses06}. For completeness, we give here a sketch of the proof.
\begin{proof}
Let us start from \begin{equation*}
 x(s)=\int_0^s u(\sigma) d\sigma= a_1\frac{\sin(\la_1 s)}{\la_1} + b_1 \frac{1-\cos(\la_1 s)}{\la_1}
 +  a_2\frac{\sin(\la_2 s)}{\la_2} + b_2 \frac{1-\cos(\la_2 s)}{\la_2}.
\end{equation*}
Elementary trigonometry gives then the form of $x(s)$ in~\eqref{essenza}.
The calculation of  $t(s)=\frac 12 \int_0^s x(\sigma)\wedge u(\s)d\s$  consists of several integrals. We calculate here the component $\pi_{ a_1\wedge b_2}t(s)$ along $\Span\{a_1\wedge b_2\}$. All other computations are similar.
\begin{equation}\label{imposto}
\begin{aligned}
 \pi_{  a_1\wedge b_2}t(s) &=\frac 12\int_0^s
 \Big\{
 \frac{\sin(\la_1\s)}{\lambda_1}\sin(\lambda_2\s) -\frac{1-\cos(\la_2\s)}{\la_2}\cos(\la_1\s)
 \Big\}d\s
\end{aligned}
\end{equation}
We have also
\begin{equation*}
\begin{aligned}
& \int_0^s\sin(\la_1\s)\sin(\la_2\s )d\s=\frac{\sin((\la_1-\la_2)s)}{2(\la_1-\la_2)}-\frac{\sin((\la_1+\la_2)s)}{2(\la_1+\la_2)},\quad\text{ and}
\\&
\int_0^s (1-\cos(\la_2\s))\cos(\la_1\s)d\s =\frac{\sin(\la_1s)}{\la_1}-
\frac{\sin((\la_1+\la_2) s)}{2(\la_1+\la_2)}
-
\frac{\sin((\la_1-\la_2) s)}{2(\la_1-\la_2)}.
\end{aligned}
\end{equation*}
Inserting into~\eqref{imposto}, we get
\begin{equation*}
 \begin{aligned}
 \pi_{  a_1\wedge b_2}t(s) &= \frac 14\Big(\frac{1}{\la_1}+\frac{1}{\la_2}\Big)
 \frac{\sin((\la_1-\la_2)s)}{\la_1-\la_2}
 +
 \frac 14\Big(\frac{1}{\la_2}-\frac{1}{\la_1}\Big)
 \frac{\sin((\la_1+\la_2)s)}{\la_1+\la_2}
 -\frac{\sin(\la_1 s)}{2\la_1\la_2}
 \\&
 =\frac 18 \Big\{\Big(\frac{1}{\phi_1}+\frac{1}{\phi_2}\Big)
 \frac{\sin((\phi_1-\phi_2)s) \cos ((\phi_1-\phi_2)s)}{\phi_1-\phi_2}
 \\&\quad +\Big(\frac{1}{\phi_2}-\frac{1}{\phi_1}\Big)
 \frac{\sin((\phi_1+\phi_2)s) \cos ((\phi_1+\phi_2)s)}{\phi_1+\phi_2}
 -\frac{2\sin(\phi_1s)\cos(\phi_1s)}{\phi_1\phi_2}
 \Big\}
 \\&=s^2G(s\phi_1,s\phi_2),
\end{aligned}
\end{equation*}
as desired.
\end{proof}


For future reference, write here  $\gamma(1,a,b,\phi)$, the extremal~\eqref{essenza} at time $s=1$.
\begin{equation}\label{estremali}
\begin{aligned}
x(1,a,b,\phi)  & =  T(\phi_1  )\Big(a_1\cos(\phi_1  )+b_1\sin(\phi_1  )\Big) +
  T(\phi_2  )\Big(a_2\cos(\phi_2  )+b_2\sin(\phi_2  )\Big)
\\
t(1,a,b,\phi) &=
 U( \phi_1)a_1\wedge b_1 +  F( \phi_1,  \phi_2) a_1\wedge a_2
 +  G( \phi_1,  \phi_2)a_1\wedge b_2
 \\ &\qquad -  G( \phi_2,  \phi_1) b_1\wedge a_2
 + H( \phi_1,  \phi_2)b_1\wedge b_2+   U( \phi_2) a_2\wedge b_2.
\end{aligned}
\end{equation}
\begin{remark}
  It can be checked that $\lim_{\phi_2\to 0} F(\phi_1, \phi_2)= \sin\phi_1 V(\phi_1)$, $\lim_{\phi_2\to 0}(-G(\phi_2, \phi_1))= -\cos\phi_1 V(\phi_1)$. Moreover,   $\lim_{\phi_2\to 0}G(\phi_1,\phi_2)=0$ and $\lim_{\phi_2\to 0}H(\phi_1, \phi_2)=0$. This means that, as $\phi_2\to 0$, then formulae~\eqref{ulisse} and~\eqref{essenza} degenerate to the known formulas in~$\F_3$, see~\cite{MontanariMorbidelli17}.
    If instead $(\phi_1, \phi_2)\to (\phi,\phi)$, where $ \phi>0$, we obtain as expected
\begin{equation*}
\begin{aligned}
(x(1),t(1))    = \big(  T(\phi   )\big( (a_1+a_2)\cos(\phi  )+(b_1+b_2)\sin(\phi   )\big)  ,
 U( \phi ) (a_1+ a_2)\wedge (b_1+b_2)\Big).
\end{aligned}
\end{equation*}
In this case, the corresponding  curve $(x(s), t(s))$ is contained in the Carnot subgroup generated by $ a_1+a_2$ and $ b_1+b_2$. See the discussion
in Subsection~\ref{pino}.
 \end{remark}

\subsection{Change of basis}
Start from \eqref{estremali} and perform the change of basis
\begin{equation}\label{svarione}
\left \{\begin{aligned}
&\a_k= a_k\sin\phi_k -b_k\cos\phi_k
\\&
\b_k= a_k\cos\phi_k +b_k\sin \phi_k
\end{aligned}\right.
 \qquad \text{i.e. }\quad  \left\{ \begin{aligned}
 & a_k=\a_k\sin\phi_k+\b_k\cos\phi_k
 \\
 & b_k=\b_k\sin\phi_k-\a_k\cos\phi_k
\end{aligned}\right.
 \end{equation}
for $k=1,2$. We are going to show that extremals become  easier  in this basis.
\begin{lemma}
Let $s_k=\sin\phi_k $ and $c_k=\cos\phi_k$ for $k=1,2$.  We have the following formulae
\begin{subequations}
\begin{align} &   \label{peppino} 2\phi_1\phi_2(\phi_1^2-\phi_2^2)F(\phi_1,\phi_2)=-(\phi_1^2+\phi_2^2)s_1c_1s_2c_2+\phi_1\phi_2(s_1^2c_2^2+c_1^2 s_2^2)
\\ &2\phi_1\phi_2(\phi_1^2-\phi_2^2)G(\phi_1,\phi_2)= -\phi_1^2s_1c_1s_2^2 +\phi_2^2s_1c_1 c_2^2
   -\phi_1\phi_2(c_1^2-s_1^2)s_2c_2
\\ &2\phi_1\phi_2(\phi_1^2-\phi_2^2)G(\phi_2,\phi_1)= -\phi_1^2c_1^2s_2c_2+\phi_2^2s_1^2s_2c_2
   +\phi_1\phi_2 s_1c_1(c_2^2-s_2^2)
\\ &2\phi_1\phi_2(\phi_1^2-\phi_2^2)H(\phi_1,\phi_2)=\phi_1^2c_1^2s_2^2+\phi_2^2s_1^2 c_2^2- 2\phi_1\phi_2s_1c_1s_2c_2.
\end{align}
\end{subequations}
\end{lemma}\label{noioso}
\begin{proof}
Let us show~\eqref{peppino} multiplied by $4$.
\begin{equation*}
\begin{aligned}
 8 &\phi_1\phi_2(\phi_1^2-\phi_2^2)F(\phi_1,\phi_2)
\\&
=\phi_1\phi_2(\phi_1^2-\phi_2^2)
 \Big\{\Big(\frac{1}{\phi_1}-\frac{1}{\phi_2}\Big)\frac{\sin^2(\phi_1+\phi_2)}{\phi_1+\phi_2}
  +
\Big(\frac{1}{\phi_1}+\frac{1}{\phi_2}\Big)
\frac{\sin^2(\phi_1-\phi_2)}{\phi_1-\phi_2}    \Big\}
\\&
=\Big\{-(\phi_1-\phi_2)^2 [s_1^2 c_2^2+ s_2^2 c_1^2+2s_1c_1s_2c_2]
+(\phi_1+\phi_2)^2[s_1^2 c_2^2+ s_2^2 c_1^2-2s_1c_1s_2c_2]\Big\}
\\&= (-\phi_1^2-\phi_2^2+2\phi_1\phi_2) [s_1^2 c_2^2+ s_2^2 c_1^2+2s_1c_1s_2c_2]
\\&\quad
+(\phi_1^2+\phi_2^2+2\phi_1\phi_2)[s_1^2 c_2^2+ s_2^2 c_1^2-2s_1c_1s_2c_2]
\\&=-4(\phi_1^2+\phi_2^2)s_1c_1s_2c_2
+4\phi_1\phi_2( s_1^2 c_2^2+ s_2^2 c_1^2).
\end{aligned}
\end{equation*}
The remaining formulas can be proved in an analogous way and we omit them.
\end{proof}

Now we are ready to prove Theorem~\ref{sempres}.
\begin{proof}[Proof of Theorem~\ref{sempres}]
 To prove the statement, start from~\eqref{estremali}  and use the change of basis \eqref{svarione}. Then $t=t(1,a,b,\phi)$ becomes
\begin{equation*}
\begin{aligned}
 t&= U(\phi_1)(\a_1 s_1+\b_1 c_1)\wedge(-\a_1c_1+\b_1 s_1)+ F(\phi_1,\phi_2)
 (\a_1 s_1+\b_1 c_1)\wedge (\a_2 s_2+\b_2 c_2)
 \\&\quad
 + G(\phi_1,\phi_2) (\a_1 s_1+\b_1 c_1)\wedge (-\a_2c_2+\b_2s_2)
 -G(\phi_2,\phi_1)(-\a_1c_1 +\b_1 s_1)\wedge (\a_2 s_2 + \b_2 c_2)
 \\& \quad
+H(\phi_1,\phi_2) (-\a_1c_1 +\b_1 s_1)\wedge (-\a_2 c_2 +\b_2 s_2)+ U(\phi_2)
(\a_2 s_2+\b_2 c_2)\wedge(-\a_2 c_2+\b_2 s_2).
\end{aligned}
\end{equation*}
The first and the last terms can be trivially written in the required form, $U_1\a_1\wedge\b_1+ U_2\a_2\wedge\b_2$,  because $(\a_k s_k+\b_k c_k)\wedge(-\a_k c_k +\b_k s_k)=\a_k\wedge\b_k$. The intermediate four terms will give contributions along $\a_1\wedge \a_2$, $\a_1\wedge\b_2$, $\b_1\wedge \a_2$ and  $\b_1\wedge \b_2$.
Let us calculate the scalar component  $\pi_{\a_1\wedge \a_2} t$ of $t$ along $\a_1\wedge\a_2$, keeping Lemma~\ref{noioso} into account.
\begin{equation*}
\begin{aligned}
 \pi_{\a_1\wedge \a_2} t &= F(\phi_1,\phi_2)s_1 s_2 -G(\phi_1,\phi_2)s_1 c_2
 +G(\phi_2,\phi_1)c_1s_2+H(\phi_1,\phi_2)c_1 c_2
 \\&=\frac{1}{2\phi_1\phi_2(\phi_1^2-\phi_2^2)}
 \Big\{  \big[-(\phi_1^2+\phi_2^2)s_1c_1s_2 c_2+\phi_1\phi_2 (s_1^2 c_2^2 + c_1^2 s_2^2)\big] s_1 s_2
 \\&\qquad \qquad \qquad  \qquad-\big [-\phi_1^2 s_1 c_1 s_2^2+\phi_2^2 s_1 c_1c_2^2
 -\phi_1\phi_2 (c_1^2 - s_1^2) s_2 c_2\big]s_1 c_2
 \\&\qquad \qquad \qquad  \qquad
 +\big[-\phi_1^2 c_1^2 s_2 c_2 + \phi_2^2 s_1^2 s_2 c_2 +\phi_1\phi_2 s_1 c_1(c_2^2 -s_2^2)
 \big] c_1 s_2
 \\&\qquad \qquad \qquad  \qquad
 +\big[ \phi_1^2 c_1^2 s_2^2 +\phi_2^2 s_1^2 c_2^2 -2\phi_1\phi_2 s_1 c_1 s_2 c_2 \big] c_1 c_2 \Big\}
 =0.
\end{aligned}
\end{equation*}
To check the last equality, it suffices to  write  $\{\cdots\}= \phi_1^2 a+\phi_2^2 b+\phi_1\phi_2 c$ and check that $a=b=c=0$ identically in $\phi_1,\phi_2$.

Let us calculate the scalar component $\pi_{\a_1\wedge\b_2}t$ of $t$ along $\a_1\wedge\b_2$. We argue as above.
\begin{equation*}
\begin{aligned}
 \pi_{\a_1\wedge \b_2} t &= F(\phi_1,\phi_2)s_1 c_2 +G(\phi_1,\phi_2)s_1 s_2
 +G(\phi_2,\phi_1)c_1c_2-H(\phi_1,\phi_2)c_1 s_2
 \\&=\frac{1}{2\phi_1\phi_2(\phi_1^2-\phi_2^2)}
  \Big\{
  \big[-(\phi_1^2+\phi_2^2)s_1c_1s_2 c_2+\phi_1\phi_2 (s_1^2 c_2^2 + c_1^2 s_2^2)\big] s_1 c_2
 \\&\qquad \qquad \qquad  \qquad
 +\big [-\phi_1^2 s_1 c_1 s_2^2+\phi_2^2 s_1 c_1c_2^2
 -\phi_1\phi_2 (c_1^2 - s_1^2) s_2 c_2\big]s_1 s_2
 \\&\qquad \qquad \qquad  \qquad
 +\big[-\phi_1^2 c_1^2 s_2 c_2 + \phi_2^2 s_1^2 s_2 c_2 +\phi_1\phi_2 s_1 c_1(c_2^2 -s_2^2)
 \big] c_1 c_2
 \\&\qquad \qquad \qquad  \qquad
 -\big[ \phi_1^2 c_1^2 s_2^2 +\phi_2^2 s_1^2 c_2^2 -2\phi_1\phi_2 s_1 c_1 s_2 c_2 \big] c_1 s_2 \Big\}.
\end{aligned}
\end{equation*}
Taking into account all  cancellations  in $\{\cdots\}$, it turns out that the terms in $\phi_2^2$ cancel and more precisely $\{\cdots\}= - \phi_1^2c_1 s_2
+\phi_1\phi_2 s_1 c_2$.
Therefore
\begin{equation*}
\begin{aligned}
 \pi_{\a_1\wedge \b_2} t &=  \frac{1}{2\phi_1\phi_2(\phi_1^2-\phi_2^2)}
   ( - \phi_1^2c_1 s_2
+\phi_1\phi_2 s_1 c_2 ) =Z(\phi_1,\phi_2),
\end{aligned}
\end{equation*}
as required.
Note that,  excanghing $2$ with $1$, we get trivially $ \pi_{\a_2\wedge\b_1}=Z(\phi_2,\phi_1)$.

We leave to the reader to check that $\pi_{\beta_1\wedge\b_2}t=0$.
\end{proof}

\begin{remark} Observe the following degenerations of the function~$Z$. For all $\phi_1 > 0$  we have
 \begin{equation}\label{azzero}
\begin{aligned}
&\lim_{\phi_2\to 0}Z(\phi_1,\phi_2)=   \frac{\sin\phi_1-\phi_1\cos\phi_1}{2\phi_1^2}=V(\phi_1)
\quad\text{ and } \quad \lim_{\phi_1\to 0}Z(\phi_1,\phi_2)= 0.
\end{aligned}
\end{equation}
Then with $\phi_2=0$ we find known formulas from~\cite{MontanariMorbidelli17}.
We also have the limit
$ Z(\phi_1,\phi_2)\mapsto  U(\phi)$,   as $ (\phi_1, \phi_2)\to (\phi,\phi)  $, for all $\phi>0$.
Recall that the function $V$ appears in~\eqref{vecchie}.
\end{remark}

Next we express any extremals at any time $s$ using the functions $U$ and $ Z$ in~\eqref{super}. Define
 \begin{equation}\label{tagliola}
 \begin{aligned}
 \a_k^s:=  s [a_k\sin(\phi_k s) -b_k\cos(\phi_k s)]  \quad\text{and}\quad \b_k^s:=   s
   [a_k\cos(\phi_k s)+ b_k\sin(\phi_k s)]
 \end{aligned}
 \end{equation}
for $k=1,2$. Then we have the following corollary.
\begin{corollary}\label{penna}
 Let $\gamma(\cdot, a,b,\phi)=(x(\cdot, a,b,\phi),t(\cdot, a,b,\phi))$ be an extremal as in Proposition~\ref{lll}. Define $\a_k^s$ and $\b_k^s$ by~\eqref{tagliola}. Then we have
\begin{equation}\label{calamo}
 \begin{aligned}
 &x(s,a,b,\phi)
 =  T(\phi_1 s)\b_1^s
 + T(\phi_2 s)\b_2^s
 \\
& t(s,a,b,\phi)  =
\\& =  U(\phi_1 s) \a_1^s \wedge \b_1^s
 + Z(\phi_1 s,\phi_2 s)\a_1^s\wedge \b_2^s
  + Z(\phi_2 s,\phi_1 s) \a_2^s\wedge \b_1^s
+  U(\phi_2 s)\a_2^s\wedge \b_2^s.
 \end{aligned}
\end{equation}
\end{corollary}
Note that if $s=1$ and $k=1,2$, then $\a_k^1=\a_k $ and $\b_k^1= \b_k$ and we recover~\eqref{super}.

 We omit the easy proof which is based on~\eqref{super},~\eqref{alfacappa} and on the reparametrization property~~\eqref{parametro}.

\begin{remark}\label{conforme}
Let us consider the form~\eqref{blocco} of $t(1,a,b,\phi)=
\Big[\begin{smallmatrix}
 0&M \\-M^T &0
\end{smallmatrix}
\Big]\in\so(4)$,
  where
$M=\begin{bmatrix}r_1^2 U_1 & r_1 r_2 Z_{12}\\ r_1 r_2 Z_{21} & r_2^2 U_2\end{bmatrix}
$.
It would be useful to understand eigenvalues and eigenspaces of $t$ in order to write it in a canonical form.  However, the eigenvalue equation takes the form
  $\la^4 + \operatorname{tr}(M^TM) \la^2 + (\det M)^2=0$,
  which becomes considerably complicated in terms of the variables $r$ and $ \phi$ contained in $M$.
There is however a subcase which seems more manageable, namely the case when $t(1)$ has double eigenvalues. Note incidentally that such points are always cut points, see the set $\Sigma_2$ in~\eqref{ciquattro}.  In view of the standard inequality
  $(\det M)^2\leq \frac 14 \big(\operatorname{tr}(M^TM) \big)^2$ for all $M\in \R^{2\times 2}$
  with equality if and only if $M $ is conformal,
 it turns out that
$t(1) $ has two double eigenvalues $i\la$ and $-i\la$ if and only if the block is conformal.  Since $U_1>0, U_2>0$ for all $\phi_1,\phi_2>0$, this gives  $r_1^2 U_1  =r_2^2 U_2$
 and $Z_{12} = - Z_{21}$.
\end{remark}

\subsection{Extremals in Carnot subgroups}\label{pino}
Next we discuss points $(x,t)$ belonging to some strict Carnot subgroup of $\F_4$. It turns out that for such points we can rely on previous known theory of length, distances and cut locus in lower rank free groups $\F_2$ and $\F_3$.
%

Let $V\subset\R^n$ be a linear subspace. Define  the Carnot subgroup generated by $V$ as $\Lie(V) := V\times \Lambda^2 V$. Note that $\Lie(V)$ is a Carnot subgroup of $\F_n$ of step $\leq 2$. Here we work with arbitrary~$n\in\N$.

\begin{proposition}\label{trentadue}
 Let $V\subsetneq \R^n$ and consider the  strict Carnot subgroup $\Lie(V)$ of~$\F_n$. Let $(x,t)\in \Lie (V)$. Let $u\in L^2(\R, \R^n)$ be  a length-minimizing control on $[0,T]$ such that $\gamma_u(0)=(0,0)$ and $\gamma_u(T)= (x,t)$.
 Then, we have $u(\R)\subset V$, or equivalently $\gamma_u(\R)\subset \Lie(V)$.
 As a consequence, we have
 \begin{equation}\label{ventuno}
  d_{\F_n}((0,0), (x,t))= d_{\Lie(V)}((0,0),(x,t))\quad\text{ for all $(x,t)\in \Lie (V)$.}
 \end{equation}
\end{proposition}
Note that the inequality $\leq $ in~\eqref{ventuno} is obvious. Equality  depends on the fact that $\F_n$ is free   (see the proof below).   As a consequence, in order to study  the cut-time of an extremal  $\gamma$ in $\Lie(V)$,  with $V\subsetneq\R^4$, it suffices to use the already known results on $\F_3$,
 and $\Cut(\Lie(V))= \Cut(\F_4)\cap \Lie(V)$.
After the proof we provide two counterexamples where equality~\eqref{ventuno} fails in nonfree settings.

\begin{proof}
 Let $(x,t)\in \Lie(V)$ be a point and let  $u(s)= :u_V^{}(s)+ u_V^\perp(s)\in V\oplus V^\perp $ be such that $\gamma_u(T)=(x,t)$. We have trivially $\length(\gamma_{u_V^{}})\leq \length \gamma_u$ with equality if and only if  $u_V^\perp(s)=0$ a.e. To conclude the proof,  we check that the control $u_V$ satisfies $\gamma_{u_V^{}}(T) = (x,t)$. Start from
 $
  x=\int_0^T u_V^{} + u_V^\perp   = : x_V^{}(T) + x_V^{\perp}(T)
 $. This implies that $\int_0^T u_V^{}=x$, as required. Moreover $\int_0^T u_V^\perp=0$. Then we look at the coordinate $t$.
 \begin{equation*}
\begin{aligned}
  t & = \frac 12\int_0^T (x_V^{}+ x_V^\perp)\wedge (u_V^{}+ u_V^\perp)
  \\& =\frac 12
  \int_0^T x_V^{} \wedge  u_V^{} + \frac 12\int_0^T (x_V^{}  \wedge     u_V^\perp + x_V^\perp\wedge u_V^{})
  +\frac 12
  \int_0^T x_V^{\perp} \wedge  u_V^\perp
  \\&
  \in\Lambda^2 V\oplus (V\wedge V^\perp)\oplus \Lambda^2 V^\perp
\end{aligned}
 \end{equation*}
where $V\wedge V^\perp=\Span\{v\wedge v^\perp: v\in V,\quad v^\perp\in V^\perp\}$ and the three subspaces $\Lambda^2 V$, $V\wedge V^\perp$ and $\Lambda^2 V^\perp$  are mutually orthogonal (this part of the argument does not generalize to nonfree settings). Since $t\in\Lambda^2 V$, we get that $\frac 12 \int x_V^{}\wedge u_V^{} = t$. Thus $\gamma_{u_V^{}}(T) = (x,t)$ and this proves the inequality $
  d_{\F_n}((0,0), (x,t))\geq  d_{\Lie(V)}((0,0),(x,t))$.
  Since $u(s)= e^{-s\t}\xi$ for suitable $\xi\in\R^n$ and $\t\in\so(n)$, it turns out by analyticity that $u(\R)\subset V$.
\end{proof}
If $\mathbb{H} $ is a Carnot subgroup of a possibly nonfree   step-2   Carnot group~$\G $, it may happen that
  $d_\G \lneqq  d_{\H}   $ at some points, where $d_\G$ and $d_\H$ denote distances from the origin.
\begin{example}   Consider the rank-4   group $\G=\R^4\times\R$ with law
\begin{equation*}
 (x,t)\cdot(\xi,\t) = \Big( x+\xi, t+\t+ \frac 12 (x\wedge \xi)_{12} + \frac{\a}{2}(x\wedge \xi)_{34}\Big)
 \in\R^4\times\R.
\end{equation*}
Take the point $(0,t)\in \Lie(e_1, e_2)$ with $t\neq 0$. It turns out that if $\a>1$, then all minimizers $\gamma$ connecting $(0,0)$ and $(0,t)\in\Lie(e_1, e_2)$
are contained in $\Lie(e_3, e_4)$   and we have $d_{\G}(0,t)=d_{\Lie(e_3, e_4)}(0,t)\lneqq
d_{\Lie(e_1, e_2)}(0,t)$.  This model has been studied in~\cite{MBarilariBoscainNeel16}.
\end{example}


  Going back to our model $\F_4$, in the following    elementary proposition, we check that a  length-minimizing curve  from the origin to
a point contained in a strict Carnot subgroup has the form $u(s)= a\cos(2\phi s)+b\sin(2\phi s)+z$, where $a,b,z\in\R^4$ are pairwise orthogonal,  $|a|=|b|>0$ and $\phi\geq 0$.

\begin{proposition}\label{giochino} Let $V\subset\R^4$ be a subspace with $\dim (V)\leq 3$.
Let $(x,t)\in  \Lie(V)$
and assume that $(x,t)=\gamma(1,a_1, b_1, a_2, b_2, \phi_1, \phi_2)$ where $\gamma$ is length-minimizing on $[0,1]$. As usual denote $|a_k|=|b_k|= r_k$ and assume also that $\phi_1\geq \phi_2\geq 0$.
 Then it must be either  $\phi_1=\phi_2\geq 0$ or,  $\phi_1>\phi_2$  and $\phi_2r_1 r_2 =0$. In other words, $\gamma$ must be non generic in the sense of Definition~\ref{generico}.
\end{proposition}\label{duu}
\begin{proof}
By Proposition~\ref{trentadue} it must be $\dim\Span\{u(s):s\in\R\}<4$.  This implies that $\dim\Span\{u(0), u'(0),
u''(0), u'''(0)\}<4$. By~\eqref{messico} we have
\begin{equation}\label{diciassette}
\begin{aligned}
&\Span\{u(0), u'(0), u''(0), u'''
(0)\}
\\&\qquad =
\Span\{a_1+ a_2, 2\phi_1 b_1 +2\phi_2 b_2, -4\phi_1^2 a_1- 4\phi_2^2 a_2,-8\phi_1^3 b_1-8\phi_2^3 b_2\}.                                                                                                       \end{aligned}
\end{equation} The span in the second line has dimension four if and only if $\gamma$ is generic. The thesis follows easily.
\end{proof}

Next we define the union $\mathcal{H}$ of all strict Carnot subgroups of~$\F_4$.
\begin{equation}
\mathcal{H} :=\bigcup\{\Lie(V): V\subset\R^4\text{ is a strict subspace}\}=\{(x,t)\in\F_4:\rank(t)\leq 2\}.
\end{equation}  To explain the equality, given $(x,t)\in\F_4$ with  $\rank(t)\leq 2$,  there are  $u,v\in\R^4$ such that
$t=u\wedge v$. This implies that $(x,t)=(x, u\wedge v)\in \Lie(\Span\{x,u,v\})\subset \mathcal{H}$. The opposite  inclusion
follows from the standard fact that if $\dim V\leq 3$, then any element  $t\in \Lambda^2 V $ is decomposable,
i.e.~$t=u\wedge v$ for suitable $u,v\in V$.

Proposition~\ref{giochino} can be rephrased as follows. If a generic extremal $\gamma$  meets the set $\mathcal{H}$, it can do it only strictly after the cut-time $t_\cut(\gamma)$.
  Recall that
 that $t_\cut<\infty$ for all such extremals,   by \cite{Kishimoto}.
\begin{proposition}\label{duuu}
 Let  $\gamma(\cdot, a,b,\phi)$ be a generic extremal (see Definition~\ref{generico}).
 Then
 \begin{equation}
  \label{quinto}
  \inf\{T>0:\rank (t(T))=2\}
  =\inf\{T>0: (x(T),t(T))\in\mathcal{H} \}  \gneqq t_{\cut}(\gamma).
 \end{equation}
\end{proposition}
\begin{proof}
Since $\gamma$ is length-minimizer on $[0,T]$ and $\gamma(T) \in \Lie(V )$, then by Proposition \ref{giochino}, $\gamma$ must be non generic, which is a contradiction.
%
%
%
\end{proof}
\begin{conjecture} \label{gettiamo} Concerning~\eqref{quinto}, we conjecture that given a generic $\gamma=\gamma(\cdot, a,b,\phi)$ it must be
 \begin{equation}\label{contumelia}
  \gamma(\left]0,+\infty\right[)\cap\mathcal{H}=\varnothing.
 \end{equation}
 In view of~\eqref{quinto} and Lemma~\ref{deboluccio} below, we have the weaker statement $\gamma
 (\left]0,+\infty\right[)\cap\mathcal{H}\subset\{ \gamma(s):s\in \left]t_\cut(\gamma),c_2(\gamma) \right[\}$ for a positive constant  $ c_2= c_2(\gamma) $
 depending on the generic extremal~$\gamma$. In Remark~\ref{messina} we translate this conjecture into an inequality.
\end{conjecture}

\section{Calculation of conjugate points  along generic extremals}\label{Piperno}

In this section we analyze conjugate points along generic extremals. Recall that conjugate points are nonzero critical points $(\xi,\t)\in T^*_{(0,0)}\F_4 \simeq\F_4$ of the map $\exp:T_{(0,0)}^*\F_4\to\F_4 $, obtained by integrating the ODE~\eqref{equino}
with control $u=u(s,\xi,\t)= e^{-s \t}\xi$
and letting $\exp(\xi,\t)=(x_u(1), t_u(1))$. In our calculations, instead of using $(\xi,\t)\in T^*_{(0,0)}\F_4$, we express the exponential map in different coordinates on a ten-dimensional manifold~$\Sigma\times\Omega$ which is
diffeomorphic to  a suitable subset of $T^*_{(0,0)}\F_4$. The preliminaries concerning such manifold will be discussed in Subsections~\ref{uffa} and~\ref{buffa}.

\subsection{Description of the manifold \texorpdfstring{$\Sigma\times\Omega $}{SXO}}\label{uffa}
 First of all, let us introduce the following notation, which will be used frequently below. Given a pair of orthonormal vectors $x,y\in\R^4$, consider the curve
$
 x(\s)=x\cos\s+y\sin\s$ and $y(\s) = -x\sin\s+y\cos \s
$, for $\s\in\R$,
which rotates $x$ and $y$ counterclockwise. Then, given a differentiable function $F(x,y)$, we introduce the notation
\begin{equation}\label{astuccio}
 D_{\circlearrowleft xy}F(x,y):=\frac{d}{d\s}F(x(\s), y(\s))\Big|_{\s=0}.
\end{equation}

Let us consider $\Sigma :=\{(x_1, x_2, x_3, x_4)\in\R^{16}: x_1, x_2, x_3, x_4 \text{  are orthonormal  in }   \R^{16} \}$. Note that $\Sigma $ is a  six-dimensional embedded submanifold, being defined by the family of ten independent equations $\langle x_j, x_k\rangle=\delta_{jk}$ for $j,k=1,\dots, 4$.
 Given $(x_1, x_2, x_3, x_4)\in\Sigma\subset\R^{16}$, we have
 \begin{equation}\label{tangente}
\begin{aligned}
T_{(x_1, x_2, x_3, x_4)}\Sigma
=\Span\{& (x_2, -x_1, 0, 0), (x_3, 0, -x_1, 0), (x_4, 0, 0, -x_1),
\\&(0, x_3, -x_2, 0), (0, x_4, 0, -x_2), (0, 0, x_4, -x_3)\} .                                                        \end{aligned}
\end{equation}
Since $x_1, x_2, x_3, x_4$ are orthogonal, the given vectors are orthogonal, then independent.
In order to see that they are tangent,
consider for any $1\leq j<k\leq 4$ the path $x^{jk}:\R\to\Sigma$ defined for all $\s\in\R$ by
$x^{jk}(\s)=(x_1^{jk}(\s), x_2^{jk}(\s), x_3^{jk}(\s),x_4^{jk}(\s)   )  $, where
\begin{equation*}
 x^{jk}_j(\s) =x_j\cos\s+x_k\sin\s, \quad x^{jk}_k(\s)= -x_j \sin\s+x_k\cos\s,\quad \text{and $x^{jk}_i(\s)=x_i$ for $i\notin\{j,k\}$.}
\end{equation*}
The set of tangent  vectors $(x^{jk})' (0)$ is described in~\eqref{tangente}.

Let also $\Omega =\{(r_1, r_2, \phi_1,\phi_2)\in\R^4: r_1, r_2>0,\; 0< \phi_2< \phi_1 \}$.
We are going to parametrize generic extremals by
\begin{equation*}
 \Gamma:\Sigma\times \Omega  \to \R^4\times\Lambda^2\R^4.
\end{equation*}
Indeed,   for $k=1,2$, write
$\a_k= r_k u_k $ and $\b_k = r_k v_k$, where $u_1, u_2, v_1, v_2$ is
an orthonormal basis in $\R^4$ and as usual $\phi_2<\phi_1$. Then,  we can write
\begin{equation*}
\begin{aligned}
&(T_1\b_1+ T_2\b_2,  U_1\a_1\wedge\b_1 + Z_{12}\a_1\wedge\b_2 + Z_{21} \a_2\wedge\b_1 + U_2\a_2\wedge \b_2)
\\=:&\Gamma(u_1, u_2, v_1, v_2, r_1, r_2, \phi_1, \phi_2)=\Gamma(u,v,r,\phi)
\\= &(r_1 T_1 v_1 +r_2 T_2 v_2, r_1^2 U_1 u_1\wedge v_1+ r_1r_2 Z_{12} u_1\wedge v_2 + r_2 r_1 Z_{21} u_2\wedge v_1 + r_2^2 U_2 u_2\wedge v_2).
\end{aligned}
\end{equation*}
To relate $\Gamma $ with $\exp$ we need the following lemma
which keeps under control the change of basis bringing $a_k,b_k$ to $\a_k, \b_k$. Writing $a_k:= r_k x_k$, $b_k= r_k y_k$, $\a_k = r_k u_k$ and $\b_k= r_k v_k$, the change $(x, y, r,\phi)\substack{\mapsto}{R}(u,v,r, \phi) $ is described in the lemma below.
\begin{lemma}\label{cambiobase}
Given the set $\Sigma\times\Omega $ defined above, consider the map $R:\Sigma\times \Omega\to \Sigma\times \Omega$ defined as
 \begin{equation*}
\begin{aligned}
&  R((x_1, y_1, x_2, y_2), (r_1, r_2, \phi_1, \phi_2))
= \big((x_1 \sin \phi_1  -y_1\cos \phi_1 , x_1\cos \phi_1 +y_1\sin\phi_1,
\\&
\qquad  \qquad \qquad \qquad \qquad
\qquad x_2 \sin \phi_2  -y_2\cos \phi_2 , x_2\cos \phi_2 +y_2\sin\phi_2),  (r_1, r_2, \phi_1, \phi_2) \big).
\end{aligned}
 \end{equation*}
Then $R$ is a diffeomorphism.
\end{lemma}
The proof is rather easy and will be presented in the appendix.

\subsection{The manifold \texorpdfstring{$\Sigma\times\Omega$}{SXO} is diffeomorphic to
  the nondegenerate part of   \texorpdfstring{$T_{(0,0)}^*\F_4$}{T*F4}}\label{buffa}
We construct a diffeomorphism $H:\Sigma\times\Omega\to G$, where
\begin{equation}\label{lampada}
G=\{(\xi,\t)\in T^*_{(0,0)} \F_4\simeq\F_4: \xi\neq 0 \text{ and $\t$ has four nonzero different eigenvalues} \}.
\end{equation}

\begin{proposition}\label{coperta}
Let $\Sigma\times\Omega$ where $\Omega =\{(r_1, r_2,   \phi_1, \phi_2  ): r_1, r_2>0,\;0<   \phi_2<\phi_1 \}$. Let also
$G\subset\F_4$ be the set defined above. Then,  the pair of requirements
\begin{equation}\label{diffe}
\left\{\begin{aligned}
 & \xi= r_1 x_1+r_2 x_2
 \\& \t=   2\phi_1 x_1\wedge y_1+2\phi_2 x_2\wedge y_2
\end{aligned}\right.
\end{equation}
defines a global diffeomorphism    $(x,y,r,\phi)\in  \Sigma\times\Omega\mapsto E(x,y,r,\phi) = (\xi,\t)\in   G$,
 satisfying
\begin{equation}\label{esposto}
 e^{-s  \t}\xi = r_1[x_1\cos(2\phi_1 s)+y_1\sin(2\phi_1 s)]+r_2[x_2\cos(2\phi_2 s)+y_2\sin(2\phi_2 s)]
\end{equation}
 for all $(x,y,r,\phi)\in\Sigma\times\Omega$.
\end{proposition}
Before proving the proposition
observe the following fact.
\begin{remark}
In view of Lemma~\ref{cambiobase} and of Proposition~\ref{coperta}, letting $H:=E^{-1}: G\to \Sigma\times\Omega$,  we have
\begin{equation}\label{demos}
\begin{aligned}
 \exp(\xi, \tau)= \Gamma(R(H(\xi, \tau))), \quad\text{with}\quad H(\xi, \tau) = (x_1, y_1, x_2, y_2, r_1, r_2 ,\phi_1, \phi_2)\in\Sigma\times\Omega,
 \end{aligned}
\end{equation}
where $x_k: \frac{a_k}{  |a_k|} =\frac{a_k }{r_k}  $, $y_k=\frac{b_k }{r_k}$, while the function $(\xi,\tau)\in \R^4\times\Lambda^2\R^4\mapsto
H(\xi, \tau) =(x_1, y_1, x_2, y_2, r_1, r_2 ,\phi_1, \phi_2) \in \Sigma\times\Omega$
satisfies
\begin{equation}\label{identita}
\begin{aligned}
e^{-s\t}\xi&= r_1 [x_1\cos(2\phi_1 s)+   y_1\sin (2\phi_1 s)] + r_2
[x_2\cos(2\phi_2 s)+   y_2\sin (2\phi_2 s)]
\\&
=a_1\cos(2\phi_1  s)+   b_1\sin (2\phi_1 s) +
a_2\cos(2\phi_2 s)+   b_2\sin (2\phi_2 s)
. \end{aligned}                                     \end{equation}
Therefore, if   $(\xi,\t)\in G$,    we have that $d_{(\xi, \tau)}\exp$
is singular if and only if $d_{( u,v,r,\phi )}\Gamma$ is singular. Here $(u,v,r,\phi)=R(x,y,r,\phi)$. The following diagram can help.
\begin{equation*}\begin{tikzpicture}[every node/.style={midway}]
  \matrix[column sep={20em,between origins}, row sep={4em}] at (0,0) {
    \node(R) {$(x,y,r,\phi)\in\Sigma\times\Omega$}  ; & \node(S) {$(u,v,r,\phi)\in\Sigma\times\Omega$}; \\
    \node(R/I) {$(\xi,\t)\in G\subset T^*_{(0,0)}\F_4$}; & \node (T) {$\exp(\xi,\t)=\gamma(1,a,b,\phi)=\Gamma(u,v,r,\phi)$};\\
  };
  \draw[->] (R/I) -- (R) node[anchor=east]  {$H$};
  \draw[->] (R) -- (T) node[anchor=west]  {$\quad\gamma(1,\cdot,\cdot,\cdot)$};
  \draw[->] (R) -- (S) node[anchor=south] {$R$};
  \draw[->] (S) -- (T) node[anchor=west] {$\Gamma$};
  \draw[->] (R/I) -- (T) node[anchor=north] {$\exp$};
\end{tikzpicture}\end{equation*}
The map in the diagonal   acts as follows $(x,y,r,\phi)\mapsto \gamma(1, r_1x_1, r_1y_1, r_2x_2, r_2y_2, \phi_1,\phi_2)=\gamma(1,a_1,b_1,a_2,b_2,\phi_1,\phi_2)=\gamma(1,a,b,\phi)$.
\end{remark}

\begin{proof}[Proof of Proposition~\ref{coperta}]
The proof is articulated in four steps.

Step 0. We show first that $E (\Sigma\times\Omega )\subseteq G$.

Step 1. We show that for any given $(\xi,\t)\in G$ there is a unique    $(x,y,r,\phi)\in\Sigma\times \Omega$   such that~\eqref{diffe} holds.

Step 2. We show that the differential of $E$ is nonsingular at any point   $(x,y,r,\phi)$.

Step 3. We show that~\eqref{esposto} holds.

Step 0  follows from the fact that $r_1, r_2>0$ which implies $\xi\neq 0$. Furthermore, it is easy to check for $k=1,2$, that $x_k\pm i y_k$  are eigenvectors corresponding to   $\pm  2i\phi_k$ of the matrix   $\t = 2\phi_1 (x_1 y_1^T - y_1 x_1^T)+2\phi_2 (x_2 y_2^T-y_2 x_2^T)$. Thus $\t$ has four different nonzero  eigenvalues $\pm 2i\phi_1$ and $\pm 2i\phi_2$.

Let us prove Step 1. Given $(\xi, \t)\in G$, by definition of $G$, the matrix $\t$ has four nonzero different eigenvalues. Thus we find unique  positive numbers    $\phi_1,\phi_2$ such that $\phi_2 <\phi_1$ and
$\pm 2i\phi_k $    are the eigenvalues of $\t$.  Let  $\bar x_k\pm i \bar y_k$ be an eigenvector corresponding to    $\pm 2i\phi_k$.   By standard properties of antisymmetric matrices, it must be $|\bar x_k|=|\bar y_k|$ and
$\langle \bar x_k, \bar y_k\rangle=0$.\footnote{If $x+iy\in\C^n$ is eigenvector of $A\in\mathfrak{so}(n)$ with eigenvalue $i\la\neq 0$, then we have $Ax=-\la y$ and $Ay=\la x$. Thus we have
$ 0=\langle Ax, x\rangle =-\la\langle y,x\rangle$ and
$
\la^2 |x|^2=\langle\la^2 x, x\rangle =\langle-A^2 x, x\rangle=\langle Ax,Ax\rangle =\langle -\la y,-\la y\rangle=\lambda^2|y|^2.
$
} Requiring also that all $\bar x_k, \bar y_k$ have unit norm in~$\R^4$,
the eigenvector $\bar x_k+i\bar y_k$ is uniquely determined up to a   rotation of the form $(\bar x_k, \bar y_k)\mapsto(x_k,y_k) = (\bar x_k\cos\theta_k +\bar y_k\sin\theta_k, -\bar x_k\sin\theta_k+\bar y_k\cos\theta_k) $.
Requirement in the first line of~\eqref{diffe} gives uniquely the choice of $\theta_k$, namely the unique choice making  $\langle y_k,\xi\rangle =0$ and $\langle x_k,\xi\rangle>0$ for $k=1,2$. Then,  letting $r_k = \langle x_k, \xi \rangle$, we find $\xi = r_1 x_1+r_2x_2$.

Let us pass to Step 2. To calculate the differential of   $E$, we need to differentiate $E$ on the manifold $\Sigma\times\Omega$. The first six columns of the matrix below contain derivatives along the tangent space to $\Sigma $ at $(x_1, y_1, x_2, y_2)$.  Notation $D_{\circlearrowleft x_1 y_1}$ and similar are explained in~\eqref{astuccio}.
In the first row  $\pi_{x_j}, \pi_{x_j\wedge y_k}$ and similar symbols denote components along $x_j, x_j\wedge y_k$ and so on.
Ultimately the differential   is represented by the following matrix.
\begin{equation*}
  \begin{aligned}
\arraycolsep=3pt\def\arraystretch{1.2}
    \begin{NiceArray}{l \left[cccccccccc\right]}[first-row]   &    D_{\circlearrowleft x_1 y_1} & D_{\circlearrowleft x_1 x_2} & D_{\circlearrowleft x_1 y_2}
 & D_{\circlearrowleft y_1 x_2} & D_{\circlearrowleft y_1 y_2} & D_{\circlearrowleft x_2 y_2}  & \p_{r_1}
 &\p_{r_2} &   \p_{\phi_1}&\p_{\phi_2}
    \\
 \pi_{x_1} & 0& -r_2 &0&0&0&0& 1&0&0 &0
\\
\pi_{y_1} & r_1& 0 &0&-r_2& 0&0& 0&0&0 &0
  \\
\pi_{x_2}  &  0 & r_1 & 0&0&0&0& 0&1&0&0
 \\
\pi_{y_2} & 0&0&r_1&0&0&r_2& 0&0&0&0
 \\
\pi_{x_1\wedge y_1} &0 &0& 0&0&0&0&0&0&2&0
 \\
\pi_{x_1\wedge x_2}  & 0&0&2\phi_2  &2\phi_1 & 0&0&0&0&0&0
 \\
\pi_{x_1\wedge y_2}  & 0&-2\phi_2 & 0 & 0 & 2\phi_1  & 0&0&0&0&0
 \\
\pi_{y_1\wedge x_2} & 0& -2\phi_1 & 0&0& 2\phi_2 & 0&0&0&0&0
 \\
\pi_{y_1\wedge y_2}  & 0&0&-2\phi_1&-2\phi_2 & 0&0&0&0&0&0
 \\
\pi_{x_2\wedge y_2\hspace{.5em}}   & 0&0&0&0&0&0&0&0&0&2
    \end{NiceArray}
\end{aligned}.
\end{equation*}
It is easy to see that calling $W_k$ the $k-$th column, we have $\Span \{W_1, W_6, W_7, W_8, W_9, W_{10}\}=
\Span\{u_1, v_1, u_2, v_2, u_1\wedge v_1, u_2\wedge v_2\}$. Thus, to check that the matrix has full rank it suffices to check that the square matrix
\begin{equation*}
\left[ \begin{smallmatrix}
0&2\phi_2&2\phi_1
&0
\\ \\ -2\phi_2 & 0 & 0 & 2\phi_1
\\   \\ -2\phi_1 & 0 & 0 & 2\phi_2
\\ \\ 0 & -2\phi_1 & -2\phi_2 & 0
 \end{smallmatrix}\right]
\end{equation*}
has full rank, which is true, because  $0<2\phi_2 < 2\phi_1 $.

To conclude the proof we prove Step 3. First of all it is easy to check that, under~\eqref{diffe} we have
\begin{equation}\label{essilor}
 e^{-s  \t } = \sum_{k=1}^2\Big\{  (x_k x_k^T + y_k y_k^T)\cos(2\phi_k s) -(x_ky_k^T - y_k x_k^T)\sin(2\phi_k s)\Big\}
\end{equation}
(the right and left-hand side have the same $\frac{d}{ds}$- derivative and agree at $s=0$). To prove~\eqref{esposto}, it suffices to multiply~\eqref{essilor} with $\xi = r_1 x_1+ r_2 x_2$.
\end{proof}

  To conclude this preliminary discussion, we observe that for $(\xi,\t)\in G\subset\F_4$,
$\exp(\xi, \t)$ is conjugate to the origin along the trajectory of the control $u(s)=e^{-s\t}\xi$ if and only if
$d_{R(H(\xi,\t))}\Gamma =d_{(u,v,r,\phi)}\Gamma $ is singular.

\subsection{Calculation of the differential of \texorpdfstring{$\Gamma$}{Gamma}} Next, in order to get information on the conjugate locus, we calculate explicitly
 the differential of $\Gamma$.
We start from
\begin{equation}\label{esselunga}
\begin{aligned}
& \Gamma(u_1, v_1, u_2, v_2, r_1, r_2, \phi_1, \phi_2)
=\Gamma(u,v,r,\phi)
\\& =\Big( r_1 T_1v_1+r_2 T_2 v_2 ,
r_1^2 U_1 u_1\wedge v_1 + r_1 r_2 Z_{12} u_1\wedge v_2 + r_1 r_2 Z_{21} u_2\wedge v_1 + r_2^2 U_2 u_2\wedge v_2\Big),
\end{aligned}
\end{equation}
where we recall that $\Gamma(u,v,r,\phi)=\gamma(1,a,b,\phi)$.
\begin{theorem}\label{magno}
Let  $0<\phi_2<\phi_1 $, Let also $r_1>0$, $r_2>0$ and $(u_1, v_1, u_2, v_2)\in\Sigma$. The point~$\Gamma(u,v,r,\phi)=\gamma(1,a,b,\phi)$ is conjugate
to the origin along $\gamma(\cdot, a,b,\phi)$ if and only if the following matrix   is singular.
 \footnotesize
\begin{equation}
  \begin{aligned}
\arraycolsep=1.4pt\def\arraystretch{1.4}
    \begin{NiceArray}{l \left[cccccccccc\right]}[first-row]
      & D_{\circlearrowleft u_1 v_1} & D_{\circlearrowleft u_2v_2} &
 D_{\circlearrowleft u_1 v_2} &  D_{\circlearrowleft u_2 v_1} &
      D_{\circlearrowleft u_1 u_2}   & D_{\circlearrowleft v_1 v_2}  &
      \partial /\partial_{r_1}   & \partial  /\partial_{r_2} & \p/\p {\phi_1} &\p/\p {\phi_2}
      \\
\pi_{ \a_1}   &  -T_1 & 0 & -r_2^2 T_2 & 0 & 0 & 0 & 0 & 0 & 0 & 0
\\
\pi_{\b_1}   & 0&   0 & 0 & 0 &0 & -r_2^2 T_2 & T_1 & 0 & -2V_1 & 0
\\
\pi_ {\a_2}   &  0 & - T_2 & 0 & -r_1^2T_1 &0 & 0 & 0 & 0 & 0 & 0
\\
\pi_{ \b_2 } &  0&0&0&0 &0 & r_1^2 T_1 & 0 & T_2 & 0 & -2V_2
 \\
 \pi_{{\a_1\wedge \b_1 }}\hspace{.5em} \quad   &  0&0&0&0& -r_2^2 Z_{21}& -r_2^2Z_{12} & 2U_1 & 0  &  \frac{\cos \phi_1}{\phi_1} V_1 & 0
 \\
 \pi_{ \a_1\wedge \a_2}   &
Z_{21} & - Z_{12} &      r_2^2 U_2  &  -r_1^2 U_1 &0 & 0 & 0 & 0 & 0 & 0
  \\
\pi_{\a_1\wedge  \b_2  } & 0&0& 0&0 & -r_2^2 U_2 & r_1^2 U_1 & Z_{12} & Z_{12} & (\p_1Z)  (\phi_1, \phi_2) &   (  \p_2 Z )(\phi_1, \phi_2)
\\
\pi_{\b_1\wedge \a_2} &  0&0&0&0& -r_1^2U_1 & r_2^2 U_2 & -Z_{21} & -Z_{21} & - (\p_2 Z)(\phi_2, \phi_1)   & -(\p_1 Z) (\phi_2, \phi_1)
\\
\pi_{\b_1\wedge \b_2} & Z_{12}&-Z_{21}  & - r_1^2   U_1  &
r_2^2 U_2 & 0 & 0 & 0 & 0 & 0 & 0
\\
\pi_{\a_2\wedge \b_2} & 0 & 0&0&0 & r_1^2 Z_{12} & r_1^2 Z_{21} & 0 & 2U_2 & 0 &\frac{\cos\phi_2}{\phi_2}V_2
    \end{NiceArray}
\end{aligned}. \label{dieci}
\end{equation}
\normalsize
%
%
%
Equivalently, $\Gamma(u,v,r,\phi) $  is conjugate if and only if at least one of the following two matrices is singular:
\begin{equation}
\label{M1}
M_1 =
\left[\begin{smallmatrix}
   -T_1 & 0 & -r_2^2 T_2 & 0
\\ \\
   0 & - T_2 & 0 & -r_1^2T_1
\\ \\
Z_{21} & - Z_{12} &      r_2^2 U_2  &  -r_1^2 U_1
\\ \\
 Z_{12}&-Z_{21}  & - r_1^2   U_1  &
r_2^2 U_2
\end{smallmatrix}\right]
\end{equation}
or
\begin{equation}\label{seipersei}
 M_2=\left[\begin{smallmatrix}
 0 & -r_2^2 T_2 & T_1 & 0 & -2V_1 & 0
\\ \\
    0 & r_1^2 T_1 & 0 & T_2 & 0 & -2V_2
\\ \\
     -r_2^2 Z_{21}& -r_2^2Z_{12} & 2U_1 & 0  &  \frac{\cos \phi_1}{\phi_1} V_1 & 0
 \\ \\
     -r_2^2 U_2 & r_1^2 U_1 & Z_{12} & Z_{12} &  (\p_1 Z)(\phi_1, \phi_2) &    (\p_2 Z)(\phi_1, \phi_2)
\\ \\
     -r_1^2U_1 & r_2^2 U_2 & -Z_{21} & -Z_{21} & -( \p_2 Z)(\phi_2, \phi_1) & - (\p_1Z)(\phi_2, \phi_1)
\\ \\
    r_1^2 Z_{12} & r_1^2 Z_{21} & 0 & 2U_2 & 0 &\frac{\cos\phi_2}{\phi_2}V_2
\end{smallmatrix}\right].
\end{equation}
\end{theorem}

In~\eqref{dieci}
 $ \p_1 Z  $ and $\p_2 Z$ denote derivatives with respect to the first and the second argument.
    Symbols $\pi_{\a_k}$, $\pi_{\a_j\wedge\b_k}$ and similar denote projections along $\a_k= r_ku_k$ and $\b_k= r_k v_k$ for $k=1,2$.
 We also denoted $V(\phi)=\frac{\sin\phi -\phi\cos\phi}{2\phi^2}$.

 Note that singularity
 of $M_1$ and $M_2$ depends on the four variables $r_1, r_2, \phi_1$ and $\phi_2$ but not on variables on $\Sigma$, by rotation invariance.

 The proof of  Theorem~\ref{magno} is postponed to the appendix.

\section{Analysis of conjugate points with \texorpdfstring{$\det M_1=0$}{det M1=0}}\label{quartina}
This section is devoted to the analysis of properties of conjugate points coming from the factor $\det M_1=0$.
\begin{lemma}
Let $(u,v,r,\phi)\in\Sigma\times\Omega$ where $r_1, r_2>0$ and $0<\phi_2<\phi_1$. Let~$M_1$ be the matrix in~\eqref{M1}. Then,  if $\sin\phi_1=\sin\phi_2=0$,  then $\det M_1=0$. If $\sin^2\phi_1+\sin^2\phi_2>0$, then $M_1$ is singular if and only if
 \begin{equation} \label{quattrosei}
  \det\left[ \begin{matrix}
r_2^2(T_2Z_{21}-T_1U_2) & r_1^2(T_1 Z_{12}-T_2 U_1)
\\  r_2^2 T_2Z_{12}+ r_1^2 T_1 U_1 & r_1^2 T_1 Z_{21} + r_2^2 T_2 U_2
       \end{matrix}\right]=0.
 \end{equation}
\end{lemma}

\begin{proof}
 If $\sin\phi_1=\sin\phi_2=0$, then $T_1=T_2=0$, so that $M_1$ is singular. Assume now that $\sin\phi_2\neq 0$, which implies $T_2\neq 0$. Changing the first column $C_1$ with $r_2^2T_2 C_1 - T_1 C_3$ and the fourth with $-r_1^2 T_1 C_2+T_2 C_4$ we get that
\begin{equation*}
 M_1\sim \left[ \begin{smallmatrix}
 0 & 0 & -r_2^2T_2 &  0
\\ \\ 0 &-T_2 & 0 &  0
 \\ \\
 r_2^2(T_2Z_{21}-T_1U_2) & - Z_{12} & r_2^2 U_2 & r_1^2(T_1 Z_{12}-T_2 U_1)
\\   \\ r_2^2 T_2Z_{12}+ r_1^2 T_1 U_1 &  -Z_{21} & -r_1^2U_1 & r_1^2 T_1 Z_{21} + r_2^2 T_2 U_2
       \end{smallmatrix}\right]
\end{equation*}
and the determinant~\eqref{quattrosei} appears.

If instead $\sin\phi_1\neq 0 $ we change $C_3\mapsto T_1 C_3 -r_2^2 T_2 C_1$ and $C_2\mapsto r_1^2 T_1 C_2 - T_2 C_4$. After some computation, we discover that the determinant is the same.
\end{proof}

\begin{remark}[Degeneration to $\F_3$] \label{aosta} The determinant~\eqref{quattrosei} degenerates correctly if $\phi_2 \to 0$. Indeed,  keeping the
    limits~\eqref{azzero} into account,
  the point $(x,t)$ becomes $(x,t)=(r_1T_1 v_1+r_2 v_2, r_1^2 U_1 u_1\wedge v_1 + r_1 r_2 V_1 u_1\wedge v_2)$, which is the form of points along extremal curves in $\Lie(u_1, v_1, v_2)$, see~\cite{MontanariMorbidelli17}. Furthermore, \ref{quattrosei} takes the form
$
   r_1^2(  U_1-T_1 V_1)(r_2^2 V_1 + r_1^2 T_1 U_1)=0 $.
 Since $U_1-T_1 V_1>0$ for all $\phi_1>0$,  see \cite[Lemma~3.1]{MontanariMorbidelli17}, it must be $\frac{r_2^2}{r_1^2}=-\frac{T_1 U_1}{V_1}$, compare \cite[Theorem~4.2]{MontanariMorbidelli17}.
\end{remark}

\begin{theorem}\label{invernizzi}
 Let $(u,v)\in\Sigma$ and consider $r_1, r_2>0$ and $0<\phi_2< \phi_1$.
 Consider  the corresponding generic extremal point  $(x,t):=\Gamma(u,v,r,\phi)$ appearing in~\eqref{esselunga}.
Then, the following two properties are equivalent:
\begin{equation}\label{primavera}
\det M_1 =0
\end{equation}
and
\begin{equation}\label{autunno}
  t^2 x \in\Span\{x\}.
\end{equation}
\end{theorem}

\begin{proof} In the basis $u_1, u_2,v_1, v_2 $ we have
\begin{equation}\label{maio}
  t=
 \begin{bmatrix}
0 & 0  & r_1^2 U_1 & r_1 r_2 Z_{12}
\\
0& 0 &  r_1 r_2Z_{21}& r_2^2 U_2
\\
-r_1^2  U_1  & -r_1 r_2 Z_{21} & 0 &0
\\
-r_1 r_2 Z_{12} & -r_2^2 U_2 &0 & 0
\end{bmatrix} =:\begin{bmatrix}
0&M\\-M^T&0
    \end{bmatrix}
    \quad\text{ and }
    x=\begin{bmatrix}
0\\0\\r_1 T_1\\ r_2 T_2
   \end{bmatrix}.
\end{equation}
By the block structure of $ t $, requiring
 $  t^2 x\in\Span\{ x\}$ is the same of requiring
 $\langle  tx ,  ty\rangle=0$, where $y= (0, 0, -r_2T_2, r_1 T_1)^T
\perp x$  in $\Span\{v_1, v_2\}$.
The calculation of $tx$ and $ty$ gives the results
\begin{equation}\label{tagliere}
\begin{aligned}
tx=\begin{bmatrix}
r_1(r_1^2T_1 U_1+  r_2^2T_2Z_{12})
\\ r_2( r_1^2   T_1 Z_{21}+r_2^2 T_2 U_2)
\\ 0 \\ 0\end{bmatrix}
\quad \text{ and } ty=
r_1 r_2 \begin{bmatrix} r_1(T_1 Z_{12}-T_2 U_1)
 \\ r_2(T_1 U_2 - T_2 Z_{21})
 \\ 0
 \\ 0
\end{bmatrix}.
\end{aligned}
\end{equation}
Requiring orthogonality between these vectors is equivalent to~\eqref{quattrosei}.
\end{proof}
\begin{remark} A calculation of the determinant~\eqref{quattrosei} shows that the equation $\det M_1=0$ can be written in the form
\begin{equation}\label{AB2}
 A(\phi_1, \phi_2)r_1^4 + \{ B(\phi_1, \phi_2)- B(\phi_2, \phi_1) \}r_1^2 r_2^2
  -A(\phi_2, \phi_1)r_2^4=0
\end{equation}
where
$
 A(\phi_1, \phi_2)=T_1 U_1 \big(T_1Z_{12}-T_2 U_1\big)$, $ B(\phi_1, \phi_2)= T_2 Z_{12} \big(T_1Z_{12} - T_2 U_1\big)$. Equation~\eqref{AB2} can be also seen as a quadratic equation in $r_1^2/r_2^2$.
 \end{remark}

\begin{remark}\label{competizione}
It is easy to see that, if $(x,t)\in \F_4$ and $\rank(t)=4$, i.e. $t$ has maximal rank, then we have
\begin{equation}\label{gara}
 t^2 x\in\Span\{x\}\qquad\Leftrightarrow\qquad
 (x,t)\in C_4.                               \end{equation}
 Indeed, if $\rank(t)=4$, then we have the equivalence $t=\la_1 u_1\wedge v_1 +\la_2 u_2\wedge v_2$ for suitable $(u_1, v_1, u_2, v_2)\in\Sigma$ and $ \la_1, \la_2>0$. If $\la_1=\la_2=:\la\neq 0$, then $(x,t)\in \Sigma_2\subset C_4$, see~\eqref{ciquattro}. We also have  $t^2 = -\la^2 I_4$ and condition $ t^2 x\in\Span\{x\}$ is obvious. Let now $0<\la_2<\la_1$. Condition $t^2 x\in\Span\{x\}$ is equivalent to claim that $x$ is an eigenvalue of \begin{equation*}
t^2=-\la_1^2 (u_1 u_1^T+v_1 v_1^T) -\la_2^2 (u_2 u_2^T+v_2 v_2^T)=-\la_1^2\pi_{\Span\{u_1, v_1\}} -\la_2^2
\pi_{\Span\{u_2, v_2\}}.                                                                                                                \end{equation*}
In other words, $x\in \Span\{u_1, v_1\}\bigcup  \Span\{u_2, v_2 \}$, which means $(x,t)\in\Sigma_1$, see again~\eqref{ciquattro}.

Finally, note that if $\rank(t)=2$, then equivalence~\eqref{gara} does not hold. See the    example $(x,t)= (e_1, e_1\wedge e_2)\in \Lie (\Span\{e_1, e_2\})$, and we have  $t^2 x\in\Span\{x\}$, but $(x,t)\notin \Cut(\F_4)$, because $\Cut(\F_4)\cap \Lie(\Span\{e_1, e_2\})= \Cut(\Lie(\Span\{e_1, e_2\}))$, by the discussion in Subsection~\ref{pino}. By Heisenberg group properties, we have $(e_1, e_1\wedge e_2)\notin \Cut(\Lie(\Span\{e_1, e_2\}))$.
\end{remark}

In~\eqref{quinto} we proved that if $\gamma(\cdot, a,b,\phi)$ is a generic extremal, then $\rank (t(s, a,b,\phi))=4$ for all $s\in \left]0,t_\cut(\gamma)\right]$. Next we prove that the same happens for large times.
\begin{lemma}\label{deboluccio}
We have the following facts.

(1) Let $\gamma(\cdot, a, b, \phi  )$  be a generic extremal.
Then, if
\begin{equation}\label{metheny}
  \phi_2>1\quad\text{ and }\quad  \phi_1 >2+\phi_2+\frac{2}{\phi_2-1}
\end{equation}
we have $\rank(t(1,a,b,\phi))=4$.

(2) For all  generic extremal $(x(s), t(s)):= \gamma (s, a, b , \phi  )$
there is
 $s^*=s^*(\phi_1,\phi_2)>0$  such that $\rank (t(s))=4$ for all $s\geq s^*$.
\end{lemma}
Note that the constant $s^*$ in (2) depends on $\phi_1$ and $\phi_2$ only, not on $a, b $.
\begin{proof}
We prove part (1). Write $\gamma(1,a,b,\phi)=\Gamma(r,u,v,\phi)$ as in~\eqref{arezzo}.
Keeping~\eqref{maio} into account, we must prove the inequality  $\det \begin{bmatrix}
 r_1^2 U_1 & r_1 r_2Z_{12} \\ r_1 r_2 Z_{21} & r_2^2 U_2
 \end{bmatrix} \neq 0
$, which means
\begin{equation}\label{bari}
 U_1U_2-Z_{12}Z_{21}\neq 0,
 \end{equation}   for all $0<\phi_2<\phi_1$ satisfying ~\eqref{metheny}.
 Equivalently,  \begin{equation}\label{palermo}
  (\phi_1^2-\phi_2^2)^2 \big(\phi_1-c_1s_1\big)\big(\phi_2-c_2s_2\big)- 4\phi_1\phi_2
  \big( \phi_2 c_2 s_1-\phi_1 c_1 s_2\big)^2\neq 0.
   \end{equation}

We claim that the left-hand side of~\eqref{palermo} is positive for all $(\phi_1,\phi_2)$ satisfying~\eqref{metheny}. Observe the trivial bounds
\begin{equation*}
\begin{aligned}
 & (\phi_1^2-\phi_2^2)^2 \big(\phi_1-c_1s_1\big)\big(\phi_2-c_2s_2\big)
  \geq   (\phi_1 +\phi_2 )^2(\phi_1-\phi_2)^2  \big(\phi_1-1 \big)\big(\phi_2-1 \big)
\\&
4\phi_1\phi_2
  \big( \phi_2 c_2 s_1-\phi_1 c_1 s_2\big)^2 <
  4\phi_1\phi_2(\phi_1+\phi_2)^2.
\end{aligned}
   \end{equation*}
    Thus, the required inequality is ensured by $(\phi_1-\phi_2)^2
   >\frac{4\phi_1\phi_2}{(\phi_1-1)(\phi_2-1)}$.
The monotonicity
inequality $\frac{\phi_1}{\phi_1-1}<\frac{\phi_2}{\phi_2-1}$ for $1<\phi_2<\phi_1$ shows that~\eqref{palermo} holds as soon as we have $(\phi_1-\phi_2)^2 >4\frac{\phi_2^2}{(\phi_2-1)^2 }$.   Taking the square root,  this turns out to be equivalent to the second condition in~\eqref{metheny}.

 Proof of (2). By~\eqref{tagliola} and~\eqref{calamo} we have
\begin{equation*}
t(s)=  U(\phi_1 s) \a_1^s \wedge b_1^s
 + Z(\phi_1 s,\phi_2 s)\a_1^s\wedge \b_2^s
  + Z(\phi_2 s,\phi_1 s) \a_2^s\wedge \b_1^s
+  U(\phi_2 s)\a_2^s\wedge \b_2^s.
\end{equation*}
Then $t(s)$ has rank $4$ if and only if $U(\phi_1 s)U(\phi_2 s)- Z(\phi_1 s, \phi_2 s)Z(\phi_2 s , \phi_1 s)\neq 0$. By (1), this holds true provided that
\begin{equation*}
 \phi_2 s>1\quad\text{ and }\quad \phi_1 s> 2+\phi_2 s + \frac{2}{\phi_2 s -1},
\end{equation*}
and, since $0<\phi_2<\phi_1$, there is $  s^*>0$ depending on $\phi_1$ and $\phi_2$  such that both inequalities hold for all $s\geq   s^*$.
\end{proof}

\begin{corollary}\label{toyo}
If  $r_1, r_2>0$ and $0<\phi_2<\phi_1$, then, given $\gamma(\cdot, a,b,\phi)$, there is     $s^*(\gamma) \geq  t_\cut(\gamma)$ such that for any $s\in \left]0, t_\cut(\gamma)\right]\cup \left[ s^*(\gamma), +\infty\right[$  the equivalent conditions~\eqref{primavera} and~\eqref{autunno} are also equivalent to the fact that $(x,t)\in\Sigma_1\cup\Sigma_2$, where
  $\Sigma_1\cup\Sigma_2\subset C_4$,
   the conjectured cut locus, see \cite{RizziSerres16}. (The sets $\Sigma_1, \Sigma_2$ are defined in~\eqref{ciquattro}).
\end{corollary}
\begin{proof}
 Just put Remark~\ref{competizione} and Lemma~\ref{deboluccio} together.
\end{proof}

 \begin{remark}
 \label{messina}
We conjecture that inequality $U_1 U_2 -Z_{12}Z_{21}>0$ holds for all $0<\phi_2<\phi_1$. The inequality implies that the matrix $t$ in~\eqref{maio} has full rank.
As a consequence, Corollary~\ref{toyo}, holds for all $s\in\left]0,+\infty\right[$. See also the discussion in Conjecture~\ref{gettiamo}.
\end{remark}
\begin{remark}\label{Taylor}
Next we briefly show that  the inequality $U_1 U_2 -Z_{12}Z_{21}>0$ mentioned above holds for  points close to the origin. By elementary trigonometry it is easy to check that
\begin{equation}\label{Prostaferesi}
yZ(x,y)= \frac14 \left(T(x-y)-T(x+y)\right)= xZ(y,x).
\end{equation}
We also have
$
U(x)=\frac{x-\sin x\cos x}{4x^2}= \frac{1-T(2x)}{4x}
$. Then,
\[
\begin{aligned}
U(x)U(y) - Z(x,y)Z(y,x)&= \frac{(1-T(2x))(1-T(2y))-(T(x+y)-T(x-y))^2}{16xy}.
\end{aligned}\]
After a routine calculation based on Taylor's formula, one can show that
\begin{equation*}
\begin{aligned}
   U(x)U(y) - Z(x,y)Z(y,x)
=\frac{1}{3!5!}\frac 4{3\cdot 5 \cdot 7 }xy(x^2-y^2)^2+o(|(x,y)|^7) ,                                                                     \end{aligned}                                                                                                                                                           \end{equation*}
as $(x,y)\to (0,0)$. The first term is  positive, but not uniformly for  $0<y<x $   close to     $(0,0)$. In order to make the estimate uniform, we can
work for example on the set $\{(x,y): 0<y<bx\}$ for some $b<1$.

\end{remark}

\begin{remark} [Degeneration to the rank-3  case] Let us consider the generic extremal point $(x,t)= (r_1 T_1  v_1+r_2 t_2v_2,
r_1^2 U_1 u_1\wedge v_1 + r_1 r_2 Z_{12} u_1\wedge v_2 + r_1 r_2 Z_{21} u_2\wedge v_1 + r_2^2 U_2 u_2\wedge v_2)
$. Letting $\phi_2= 0$ we have the degenerations $T_2= T(0)=1$, $Z_{12}=Z(\phi_1,0)= V(\phi_1 )=V_1 $,  $Z(\phi_2, \phi_1)|_{\phi_2=0}= 0$, $U(\phi_2 )= U(0)=0$. Therefore, we get the extremal point
\begin{equation*}
\begin{aligned}
 (x,t)= (r_1 T_1v_1 + r_2 v_2, r_1^2 U_1 u_1\wedge v_1 + r_1 r_2 V_1 u_1\wedge v_2)
 =(T_1\b_1+ \b_2, \a_1\wedge (U_1\b_1+V_1\b_2)).
\end{aligned}
\end{equation*}
This is the general form of points  in $\Lie (u_1,v_1, v_2)\simeq \F_3$. Compare the function $G(\a,\b,\zeta,\phi)$ in \cite[Remark~2.3]{MontanariMorbidelli17}.
After some calculations, one can see that the matrix~\eqref{dieci} in the rank-3  case becomes
\small
\begin{equation}\label{undici}
  \begin{aligned}
\arraycolsep=1.4pt\def\arraystretch{1.4}
    \begin{NiceArray}{l \left[cccccc\right]}[first-row]
     & D_{\circlearrowleft u_1 v_1} &  D_{\circlearrowleft u_1 v_2} &
  D_{\circlearrowleft v_1 v_2}  &  \partial /\partial_{r_1}   & \partial  /\partial_{r_2} & \p/\p {\phi_1}
\\
\pi_{ \a_1}   &  -T_1 &   -r_2^2   & 0 & 0 & 0 & 0
\\
\pi_{\b_1}   & 0&   0 &  -r_2^2  & T_1 & 0 & -2V_1
\\
\pi_{ \b_2 } &  0&0&  r_1^2 T_1 & 0 & 1 & 0
 \\
 \pi_{{\a_1\wedge \b_1 }}  &  0&0 & -r_2^2 V_1 &  2U_1 & 0  &  \frac{\cos \phi_1}{\phi_1} V_1
\\
\pi_{\a_1\wedge  \b_2  }\quad  & 0&0& r_1^2 U_1 & V_1 & V_1 & V_1'
\\
\pi_{\b_1\wedge \b_2\hspace{.5em}} &
V_1&  - r_1^2   U_1  &
0 & 0 & 0 & 0    \end{NiceArray}
\end{aligned},
\end{equation}\normalsize
which is singular if and only if at least one among the two matrices below is singular
\begin{equation}
   N_1=
   \left[
\begin{smallmatrix}
 -T_1 & -r_2^2
 \\ V_1 & -r_1^2 U_1
\end{smallmatrix}
   \right],
\quad\text{ or }\quad N_2:=
\left[
\begin{smallmatrix}
-r_2^2 & T_1 & 0 & -2 V_1
\\ \\
r_1^2 T_1 & 0 & 1 & 0
\\ \\
- r_2^2 V_1 & 2U_1 & 0 &\frac{\cos\phi_1}{\phi_1}V_1
\\ \\ r_1^2 U_1 & V_1 & V_1 & V_1'
\end{smallmatrix}
\right].
\end{equation}
Note that the requirement $\det N_1=0$ becomes $\frac{r_2^2}{r_1^2}=- \frac{T_1U_1 }{V_1}= Q(\phi_1)$, where $Q(\phi_1)$ is the function appearing in \cite[Theorem~4.1]{MontanariMorbidelli17}. In that case, points where $\det N_1=0$ are  the points   of cut locus. Points where $\det N_2=0$ are conjugate points which likely may not belong to the cut locus. The same splitting of the critical set appears in~\cite[equation~(12)]{Myasnichenko02}. Zeros of the factor $e_1(\tau) \cos^2\phi + e_2(\tau)$ correspond to zeros of $\det N_1$ and detect cut points. Zeros of the factor $e_3(\t)\cos^2\phi + e_4(\t)$ correspond to zeros of $\det N_2$.
\end{remark}

 \section{Upper      estimates of the cut time and    Proof of Theorem~\ref{sarebbero} }\label{prima}
In this section   we give the proof of Theorem~\ref{sarebbero}. If $\phi_1$ and $\phi_2$ are rationally dependent, we also get an upper estimate of the cut time.  The proof of Theorem~\ref{sarebbero} requires more work in the rationally independent case.

In the present section, given the extremal $u(s)= \sum_{k=1}^2  a_k\cos(2\phi_k s) + b_k\sin(2\phi_k s) $, let us write~$\gamma(s)$ by formula~\eqref{calamo}. Define for  all $s $ the orthonormal vectors
 $ u_k^s:= \frac{ \a_k^s}{ sr_k} $ and $ v_k^s= \frac{\b_k^s}{sr_k} $, where we refer to~\eqref{tagliola}.
Under this notation, we have
 \begin{equation}\label{ravenna}
\begin{aligned}
 x(s, a,b,\phi)
 &= s  r_1T(\phi_1 s)v_1^s
 +sr_2 T(\phi_2 s)v_2^s
 \\
 t(s, a,b,\phi)
& =s^2 \Big( r_1^2 U(\phi_1 s) u_1^s\wedge v_1^s
 + r_1 r_2 Z(\phi_1 s,\phi_2 s)u_1^s
\wedge v_2^s
\\& \qquad\qquad + r_1 r_2 Z(\phi_2 s,\phi_1 s)u_2^s\
\wedge v_1^s +r_2^2 U(\phi_2 s)u_2^s\wedge v_2^s\Big).
 \end{aligned}
\end{equation}

In the rationally dependent case,
 we have the following easy result, partially due to Brockett.
 \begin{proposition}[Extremals with rationally dependent parameters $\phi_1$ and $\phi_2$]
 \label{brocco}
The following statements hold true.
 \begin{enumerate} \item \label{brocchetto}
 Given $(0, t)= (0, t_1 x_1\wedge y_1+ t_2 x_2\wedge y_2)$ with $t_1\geq t_2>0$ and $ x_1, y_1, x_2, y_2  $ orthonormal family in~$\R^4$, then all minimizers reaching $(0, t)$ have the form $\gamma(\cdot , a,b, \phi)$ with $\phi_2 =2\phi_1$ or $\phi_1=2\phi_2$. Moreover,
   for all $t_1\geq t_2>0$
 we have
\begin{equation}\label{bottiglia}
\begin{aligned}
 d((0,0), t_1 x_1\wedge y_1+ t_2 x_2\wedge y_2) & = \sqrt{4\pi t_1  + 8\pi t_2 }
 \\  &= \sqrt{4\pi\max\{|t_1|, |t_2|\}+ 8\pi\min\{ |t_1|, |t_2|\}\;}.
\end{aligned}
\end{equation}
 \item
Let $0<\phi_2<\phi_1$, where $\phi_1$ and $\phi_2$ are rationally dependent. Take  $r_1, r_2>0$ and consider the extremal $\gamma(\cdot, a,b,\phi)$, where $r_k=|a_k|=|b_k|>0$ for $k=1,2$. Then:
\begin{enumerate}
\item \label{ahi}  If $\phi_1=2\phi_2$ and $r_2^2\geq  \frac{r_1^2}{2}$, then we have $t_\cut(\gamma)=\frac{\pi}{\phi_2}$ and $\gamma(t_\cut)
=\big(0,
\frac{\pi}{8\phi_2^2}a_1\wedge b_1 +\frac{\pi}{4\phi_2^2}a_2\wedge b_2
\big)$.


\item \label{b} If $\phi_1=2\phi_2$ and $r_2^2<\frac{r_1^2}{2}$, then we have $t_\cut(\gamma)\lneqq \frac{\pi}{\phi_2}$

\item \label{c} If $\frac{\phi_1}{\phi_2}=\frac{p}{q}\in\Q\cap\left]1,+\infty\right[\setminus\{ 2\}$, then assuming that $p$ and $q$ do not have common divisors, we have $t_\cut(\gamma) \lneqq \frac{\pi q}{\phi_2} =\bar s(\frac{\phi_1}{\phi_2}):=\min\{s>0:s\phi_1=s\phi_2=0\; (\operatorname{mod}\pi)\}$.
\end{enumerate}
\end{enumerate}
 \end{proposition}
 In cases \eqref{b} and \eqref{c} the curve $\gamma$ reaches its cut-time before touching the vertical set $\{0\}\times \Lambda^2\R^4$.
\begin{proof}
Part~\eqref{brocchetto} is essentially contained in \cite{Brockett}. Let us recapitulate the  proof.  Let $(0, t )
 =(0, t_1 x_1\wedge y_1+ t_2 x_2\wedge y_2)$, where $(x_1, y_1, x_2, y_2)\in\Sigma$ and witout loss of generality we assume that $t_1\geq t_2>0$.
By reparametrization invariance, we may search for the shorter among all $\gamma(\cdot, a,b,\phi)$ such that $\gamma(1, a, b, \phi)=(0, t)$. Length here is  $\int_0^1|u(s)| ds=\sqrt{r_1^2+ r_2^2}$, with $r_k=|a_k|=|b_k|$.
This gives
\begin{equation*}
\left\{\begin{aligned}
 &r_1 T_1 v_1+ r_2 T_2 v_2=0
\\& r_1^2 U_1 u_1\wedge v_1 +r_1 r_2 Z_{12} u_1\wedge v_2 + r_1r_2Z_{21}u_2\wedge v_1 + U_2 u_2\wedge v_2 = t_1 x_1\wedge y_1+t_2 x_2\wedge y_2.
\end{aligned}
\right.
\end{equation*}
The first line implies that $\phi_1=\phi_2=0$ (mod $\pi$), i.e. $\phi_1=n_1\pi$ and $\phi_2= n_2\pi$, where $n_1> n_2\in\N$, if we consider as usual $\phi_2\leq \phi_1$. We must exclude $n_1=n_2$ because $\rank(t)=4$. By properties of the functions $U$ and $Z$ we obtain
\begin{equation}\label{pekk}
 \frac{r_1^2}{4n_1 \pi}u_1\wedge v_1 + \frac{r_2^2 }{4n_2\pi}u_2\wedge v_2 = t_1 x_1\wedge y_1 + t_2x_2\wedge y_2.
\end{equation}
We have then to minimize $\sqrt{r_1^2+r_2^2}$ under the  constraint given by equality~\eqref{pekk}. We are working with $n_1>n_2$.  It is easy to see that
the optimal choice if given by $n_1=2$, $n_2=1$,  $u_1\wedge v_1=x_2\wedge y_2$ and $u_2\wedge v_2= x_1\wedge y_1$.
As a consequence $r_1^2=8\pi t_2$ and $r_2^2=4\pi t_1$ and formula~\eqref{bottiglia} follows.  The proof of (1) is complete.
Note that, since $t_1\geq t_2$, we have $\frac{r_1^2}{2}\leq r_2^2$.

Next we prove ~\eqref{ahi}. Let $\phi_1=2\phi_2$ and calculate by~\eqref{ravenna} the point $\gamma(\bar s, a,b, 2\phi_2, \phi_2)$ letting $\bar s=\frac{\pi}{\phi_2}$.
\begin{equation}\label{biforcazione}
\begin{aligned}
 \gamma\Big(\frac{\pi}{\phi_2},a,b, 2\phi_2, \phi_2\Big) &=\Big(0, \frac{\pi^2}{\phi_2^2}\Big( \frac{r_1^2}{8\pi} u_1^{\bar s}\wedge v_1^{\bar s}
 +\frac{r_2^2}{4\pi} u_2^{\bar s}\wedge v_2^{\bar s}\Big)\Big)
 \\&=
 \Big(0, \frac{\pi r_1^2}{8\phi_2^2 }u_1^{\bar s}\wedge v_1^{\bar s} +
 \frac{\pi r_2^2}{4\phi_2^2} u_2^{\bar s}\wedge v_2^{\bar s}\Big)
\\& =\Big(0, \frac{\pi}{8\phi_2^2}a_1\wedge b_1 +\frac{\pi}{4\phi_2^2}a_2\wedge b_2\Big),
\end{aligned}
\end{equation}
by identity $r_k^2 u_k^s\wedge v_k^s = a_k\wedge b_k$ for all $s>0$.
  Since $r_2^2\geq \frac{r_1^2}{2}$ , \eqref{bottiglia} gives that the distance of such point from the origin is
$ \sqrt{4\pi\frac{\pi r_2^2}{4\phi_2^2}+ 8\pi\frac{\pi r_1^2}{8\phi_2^2}}=\frac{\pi}{\phi_2}\sqrt{r_1^2+ r_2^2\,}$.
This agrees with   $\length \big(\gamma|_{[0, \frac{\pi}{\phi_2}]}\big)=\int_0^{\pi/\phi_2}\sqrt{r_1^2+ r_2^2}ds$. Then $\gamma$ minimizes on $[0,\pi/\phi_2]$ and,   since $(0,t)\in C_4$ for all $t$,  we conclude that $t_\cut(\gamma)=\frac{\pi}{\phi_2}$.

We pass to the proof of \eqref{b}. We get~\eqref{biforcazione}, as in the previous case. However, here we have $r_2^2<   \frac{r_1^2}{2}$. Thus~\eqref{bottiglia} gives that the distance is
\begin{equation*}
 \sqrt{4\pi\frac{\pi r_1^2}{8\phi_2^2}
 + 8\pi \frac{\pi r_2^2}{4\phi_2^2}}=\frac{\pi}{\phi_2}\sqrt{\frac{r_1^2}{2}+ 2r_2^2\,}\lneqq \frac{\pi}{\phi_2}\sqrt{r_1^2+ r_2^2}=\length(\gamma|_{[0, \pi/\phi_2]}).
\end{equation*}
Thus $\gamma$ is not a minimizer on $[0,\pi/\phi_2]$.

Finally we show \eqref{c}. Let $\gamma\big(\cdot, a,b,\frac{p}{q} \phi_2, \phi_2\big)$, where $p,q\in\N$ and $p>q$, $p\neq 2q$ and assume $p,q$ do not have common divisors. By~\eqref{ravenna} we have
\begin{equation*}
x(s):= x\Big(s,a,b,\frac{p}{q} \phi_2, \phi_2\Big) = sr_1 T\Big(\frac{p}{q}\phi_2s\Big)  v_1^s +sr_2 T(\phi_2 s)   v_2^s.
\end{equation*}
The smallest $s>0$ such that $x(s)=0$ is $s=\frac{q\pi}{\phi_2}$. We also have $x(s)=0$ for all $s=k\frac{q\pi}{\phi_2}$ with $k\in\N$. Furthermore,  by part~\eqref{brocchetto} of the theorem,  $\gamma$  does   not minimize length on $[0, \frac{q\pi}{\phi_2}]$. Then, we have
the upper estimate $t_\cut\lneqq \frac{q\pi}{\phi_2}$.
\end{proof}

Let us pass to the analysis of extremal controls with rationally independent $\phi_1$ and $\phi_2$.
\begin{theorem}\label{sarebbe}
 Let $u(s)= \sum_{k=1}^2  a_k\cos(2\phi_k s) + b_k\sin(2\phi_k s)$ be an admissible control. Assume  also that $r_1, r_2>0$  and $\frac{\phi_1}{\phi_2}> 1$ is irrational.
 Consider the corresponding trajectory $\gamma(s,a ,b , \phi )$. Then there is a sequence $s_j\to +\infty$ such that
  $\gamma(s_j)$ is conjugate and  belongs to   $    \Sigma_1\cup\Sigma_2   \subset  C_4$ for all $j\in\N$, see  formula~\eqref{ciquattro}.
\end{theorem}
As a trivial consequence we have the following corollary.
\begin{corollary}[Kishimoto \cite{Kishimoto}]
\label{gaspare}  Let $\gamma$ be an extremal   which is not contained in a strict subgroup. Then $t_\cut(\gamma)<\infty$.
\end{corollary}
\begin{proof}[Proof of Corollary~\ref{gaspare}]
If $\frac{\phi_2}{\phi_1}$ is rational, the result is contained in Proposition~\ref{brocco}. If instead, $\frac{\phi_2}{\phi_1}$ is  irrational and $r_1 r_2>0$,  let $(s_j)_{j\in\N}$ be the sequence from Theorem~\ref{sarebbe}. Then, we have $\gamma(s_1)\in C_4\subset\Cut(\F_4)$. If $\gamma$ minimizes on $[0, s_1]$, then $t_\cut(\gamma)=s_1<\infty$. If not, then we have $t_\cut(\gamma)<s_1<\infty$.
\end{proof}

In order to prove Theorem~\ref{sarebbe}, we need  to write the equivalent conditions~\eqref{primavera} or~\eqref{autunno}
of Theorem~\ref{invernizzi} at any time $s>0$. In the statement we use the functions~$A$ and $B$ defined in~\eqref{AB2}.
\begin{lemma}\label{aab}
Let $\gamma(\cdot, a, b, \phi)$ be an extremal with $r_1, r_2>0$ and $0<\phi_2<\phi_1$. Then  $\gamma(\cdot, a,b,\phi)$ satisfies the equivalent conditions~\eqref{primavera} or~\eqref{autunno} at time~$s>0$ if and only if
\begin{equation}\label{tempaccio}
\begin{aligned}
D(s,r_1, r_2, \phi_1, \phi_2 ):=  & A(\phi_1 s,\phi_2 s)r_1^4+\{B(\phi_1s, \phi_2 s)-B(\phi_2s, \phi_1 s)\}r_1^2 r_2^2
\\& - A(\phi_2 s,\phi_1 s)r_2^4=0.
\end{aligned}
\end{equation}\end{lemma}
\begin{proof}[Proof of Lemma~\ref{aab}]
Starting from~\eqref{ravenna}, it is then easy to see that the equivalent conditions~\eqref{primavera} and~\eqref{autunno} hold if and only if~\eqref{tempaccio} holds.
  \end{proof}

\begin{proof}[Proof of Theorem~\ref{sarebbe}]  Let  $r_1, r_2>0$ and $0<\phi_2<\phi_1$,  where  $\phi_1$ and $\phi_2$ are  rationally independent.  By Lemma~\ref{deboluccio} and Lemma~\ref{aab}, we
 must prove that there is a sequence $s_j\to+\infty$ such that
$D(s_j,r,\phi )=0$ for all $j\in\N$.
%
%
 Recall that $ A(\phi_1, \phi_2)=T_1 U_1 \big(T_1Z_{12}-T_2 U_1\big)$ and
 $B(\phi_1, \phi_2)= T_2 Z_{12} \big(T_1Z_{12} - T_2 U_1\big)$.
Denote $(x,y):=(\phi_1 s, \phi_2 s)$ below. Denote also as $P_d(x,y)$ a homogeneous polynomial of degree~$d$ in~$(x,y)$ and write $h(x,y)$ for a  function  bounded in $x,y$.
\begin{equation*}
\begin{aligned}
 A(x,y) & = \frac{\sin x}{x}\Big(\frac{x-\sin x\cos x}{4x^2}\Big)\Big[\frac{\sin x}{x}\cdot \frac{y\cos y \sin x - x \cos x \sin y}{2y(x^2-y^2)}\\& \qquad \qquad \qquad
\qquad \qquad \qquad  -\frac{\sin y}{y}\Big(\frac{x-\sin x\cos x}{4x^2}\Big)\Big]
\\&
= \Big(\frac{\sin x}{4 x^2}-\frac{\sin^2 x\cos x}{4x^3} \Big)
\Big[-\frac{\sin y}{4xy} +\frac{\sin x
\cos x \sin y}{4x^2 y}
\\&\qquad\qquad \qquad\qquad \qquad\qquad +\frac{\sin x}{x} \Big( \frac{y\cos y \sin x - x \cos x \sin y}{2y(x^2-y^2)} \Big)\Big].
\end{aligned}
\end{equation*}
Organizing terms in $A(x,y)$ by the homogenity degree of the denominators we get
\begin{equation*}
\begin{aligned}
 A(x,y) & =  -\frac{\sin x\sin y}{16 x^3 y}+\frac{\sin^2 x\cos x \sin y}{8x^4 y}
 +\frac{\sin^2 x}{ 4x^3}\Big[
 \frac{y\cos y \sin x - x \cos x \sin y}{2y(x^2-y^2)}
 \Big]+\frac{h(x,y)}{P_6(x,y)}.
\end{aligned}
\end{equation*}
Let us look at $B$.
\begin{equation*}
\begin{aligned}
 B(x,y)& =\frac{\sin y}{y}
\Big( \frac{y\cos y \sin x - x \cos x \sin y}{2y(x^2-y^2)}\Big)
\cdot\Big[ \frac{\sin x}{x}
\Big(\frac{y\cos y \sin x - x \cos x \sin y}{2y(x^2-y^2)}\Big)
\\& \qquad\qquad\qquad\qquad \qquad\qquad\qquad\qquad  \qquad -
\frac{\sin y}{y}
\Big(\frac{x-\sin x\cos x}{4x^2}\Big)
\Big]
\\& =-\frac{\sin^2 y}{8xy^3}
\Big(\frac{y\cos y \sin x - x \cos x \sin y}{ x^2-y^2 }\Big)
+\frac{h(x,y)}{P_6(x,y)},
\end{aligned}
\end{equation*}
which gives
\begin{equation*}
 B(x,y)-B(y,x)= \Big(\frac{y\cos y \sin x - x \cos x \sin y}{ x^2-y^2 }\Big)\Big[\frac{\sin^2 x}{8x^3y}-\frac{\sin^2 y}{8xy^3}\Big] +\frac{h(x,y)}{P_6(x,y)}.
\end{equation*}

Let us try first to keep only terms of homogeneity $-4$.
\begin{equation}\label{smacco}
\begin{aligned}
D(s, r_1, r_2, \phi_1, \phi_2)&= A(x,y)r_1^4 +\{B(x,y)-B(y,x) \}r_1^2 r_2^2 - A(y,x)r_2^2
\\& =    \frac{\sin x\sin y}{16 x y }\Big[\frac{r_2^4}{y^2} -\frac{r_1^4}{x^2} \Big]+\frac{h(x,y, r_1, r_2)}{P_5(x,y)}+\frac{h(x,y, r_1, r_2)}{P_6(x,y)}
\\& =\frac{1}{s^4}\frac{\sin (\phi_1 s)\sin(\phi_2 s)}{16 \phi_1\phi_2   } \Big[\frac{r_2^4}{ \phi_2^2  } -\frac{r_1^4}{\phi_1^2  } \Big]+\frac{h(s)}{s^5 P_5( \phi_1,\phi_2)} +\frac{h(s)}{s^6 P_6( \phi_1,\phi_2)} ,
\end{aligned}
\end{equation}
where $h $ depends also on $(\phi_1, \phi_2, r_1, r_2)$ and is bounded globally.
At this point, if
$ \frac{r_2^4}{ \phi_2^2  } -\frac{r_1^4}{\phi_1^2  } \neq 0
$, we claim that there are sequences $s_n^+$ and $s_n^-\to +\infty$ such that
$ D(s_n^-, r_1, r_2, \phi_1, \phi_2)<0<D(s_n^+, r_1, r_2, \phi_1, \phi_2 )$ for all $n\in\N$.
This will imply that there is a sequence $s_n\to +\infty$ of zeros of $D$.
To see that, since $\phi_1$ and $\phi_2$ are rationally independent, letting $\s:=\phi_2 s$,  we get  $\sin(\phi_2 s)\sin(\phi_1 s)=\sin\s\sin(\a\s)$, where $\a\notin\Q$. Take $\s_n=\frac{\pi}{2}+ 2\pi n$, so that   $\sin \s_n=1$ for all $n\in\N$. Then, since $\a\notin\Q$, the sequence  $\sin(\a \s_n) = \sin \Big(2\pi \a n + \frac{\a\pi}{2}\Big)$ with $n\in\N$ is dense in $[-1,1]$ (by standard properties of irrational flows on torus, see~\cite[p.~26]{Jost}). Therefore, taking a subsequence $\s_n^+$ such that $\sin(\a\s_n^+)\to 1$, we have $\sin(\s_n^+)\sin(\a\s_n^+)\to 1$, a strictly positive bound. An analogous argument gives us a sequence~$\s_n^-$ such that $\sin (\s_n^-)\sin (\a\s_n ^-)\to -1$ and the claim is proved.

Let us pass now to case $ \frac{r_2^4}{ \phi_2^2  } -\frac{r_1^4}{\phi_1^2  }=0$. We must take into account the term $\frac{h(x,y)}{P_5(x,y)}$ appearing in~\eqref{smacco}. Using the expansions of $A$ and $B$ obtained above, we get
 \begin{equation*}
\begin{aligned}
 A&(x,y) r_1^4+\{B(x,y)-B(y,x)\}r_1^2 r_2^2-A(y,x) r_2^4
 \\
 &=\Big\{\frac{\sin^2 x\cos x\sin y}{8x^4 y}
 +\frac{\sin^2 x}{8x^3 y}
 \Big( \frac{y\cos y \sin x - x \cos x \sin y}{ x^2-y^2 } \Big)
 \Big\}r_1^4
 \\
 &\quad + \Big\{
 \frac{y\cos y \sin x - x \cos x \sin y}{  x^2-y^2 }
 \Big\}
 \Big[\frac{\sin^2 x}{8x^3 y}-\frac{\sin^2 y}{8x y^3}\Big]
 r_1^2 r_2^2
 \\&\qquad
 - \Big\{\frac{\sin x \sin^2 y\cos y}{8x  y^4}
 +\frac{\sin^2 y}{8x  y^3}
 \Big( \frac{y\cos y \sin x - x \cos x \sin y}{ x^2-y^2 } \Big)
 \Big\}r_2^4 +\frac{h(x,y,r_1, r_2)}{P_6(x,y)}.
\end{aligned}
\end{equation*}
Passing to   $(x,y)=(\phi_1 s, \phi_2 s) $ and writing  $ \phi_k^s:=s\phi_k $ we get
\begin{equation*}\begin{aligned}
&  D(s , r_1,r_2,\phi_1,\phi_2)
 \\&=\Big\{
 \frac{\sin^2 \phi_1^s \cos \phi_1^s  \sin \phi_2^s }{8\phi_1^4\phi_2s^5}
 +\frac{\sin^2 \phi_1^s }{8\phi_1^3\phi_2 s^4 }
 \Big(\frac{\phi_2\cos\phi_2^s\sin\phi_1^s -\phi_1\cos\phi_1^s\sin\phi_2^s}{s(\phi_1^2-\phi_2^2)}\Big)
 \Big\}r_1^4
 \\&
 +\Big(\frac{\phi_2\cos\phi_2^s\sin\phi_1^s -\phi_1\cos\phi_1^s\sin\phi_2^s}{s(\phi_1^2-\phi_2^2)}\Big)
 \Big\}
 \Big[\frac{\sin^2\phi_1^s}{8\phi_1^3\phi_2 s^4}-
 \frac{\sin^2\phi_2^s}{8\phi_1\phi_2^3  s^4}\Big] r_1^2 r_2^2
 \\&
 -\Big\{\frac{\sin\phi_1^s\sin^2\phi_2^s\cos\phi_2^s}{8\phi_1\phi_2^4 s^5}
+ \frac{\sin^2\phi_2^s}{8\phi_1\phi_2^3s^4}
  \Big(\frac{\phi_2\cos\phi_2^s\sin\phi_1^s -\phi_1\cos\phi_1^s\sin\phi_2^s}{s(\phi_1^2-\phi_2^2)}\Big)
 \Big\}r_2^4+\frac{h(s,\phi,r)}{s^6P_6(\phi,r)}.
\end{aligned}
\end{equation*}
Multiply by  $\frac{\phi_1^2}{r_1^4}$ and eliminate  $r_2$ by $r_2^2=\frac{\phi_2}{\phi_1}r_1^2$.
\begin{equation*}
\begin{aligned}
& \frac{\phi_1^2}{r_1^4} D(s , r_1,r_2,\phi_1,\phi_2)
\\ = &\frac{1}{s^5}
\bigg\{\frac{\sin^2\phi_1^s\cos\phi_1^s\sin\phi_2^s}{8\phi_1^2\phi_2}
+ \frac{\sin^2\phi_1^s}{8\phi_1\phi_2}
 \Big(\frac{\phi_2\cos\phi_2^s\sin\phi_1^s -\phi_1\cos\phi_1^s\sin\phi_2^s}{\phi_1^2-\phi_2^2}\Big)
\\&\quad
  +\Big(\frac{\phi_2\cos\phi_2^s\sin\phi_1^s -\phi_1\cos\phi_1^s\sin\phi_2^s}{\phi_1^2-\phi_2^2}\Big)
\Big[\frac{\sin^2\phi_1^s}{8\phi_1^2}-
\frac{\sin^2\phi_2^s}{8\phi_2^2}\Big]
\\&\quad
  -
\frac{\sin \phi_1^s\sin^2\phi_2^s\cos\phi_2^s }{8\phi_1\phi_2^2}
- \frac{\sin^2\phi_2^s}{8\phi_1\phi_2}
 \Big(\frac{\phi_2\cos\phi_2^s\sin\phi_1^s -\phi_1\cos\phi_1^s\sin\phi_2^s}{\phi_1^2-\phi_2^2}\Big)
\bigg\}
+\frac{h(s,\phi_1, \phi_2)}{s^6P_4(\phi_1, \phi_2)}.
\end{aligned}
\end{equation*}
Take now the sequence $s_n=\frac{n\pi}{\phi_2}$, so that  $\sin(\phi_2 s_n)=0 $  and  we get
\begin{equation*}
\begin{aligned}
& \frac{\phi_1^2}{r_1^4}D(s_n,r_1, r_2, \phi_1, \phi_2)
 %
 \\&=\frac{1}{s_n^5}\cdot\frac{\phi_2}{8\phi_1(\phi_1^2-\phi_2^2)}\Big(\frac{1}{\phi_2}+\frac{1}{\phi_1}\Big)\cos(\phi_2 s_n)\sin^3(\phi_1s_n)+\frac{h(s_n,\phi_1,\phi_2)}{s_n^6 P_4(\phi_1,\phi_2)}.
 \end{aligned}
\end{equation*}
Since we are assuming $\phi_1>\phi_2$, it turns out that the sign of $D$, for large $n$ is the same of
\begin{equation*}
 \sin^3(\phi_1 s_n)\cos(\phi_2 s_n)= (-1)^{n}\sin^3 \Big(\frac{\phi_1}{\phi_2}n\pi\Big).
\end{equation*}
We must find two subsequences, one converging  to a positive limit and the other to a negative one.
Since $\phi_1$ and $\phi_2$ are  rationally independent, we  use again the standard fact that for all  $\a\notin\mathbb{Q}$ we have  $\liminf_{n\to +\infty} \sin (\a n\pi)=-1$ and $\limsup_n\sin(\a n\pi)=+1$ (see again~\cite{Jost}).    The proof is easily concluded.
\end{proof}

\appendix
 \section{}

\begin{proof}[Proof of Lemma~\ref{cambiobase}]
 Denoting $c_k = \cos \phi_k$ and $s_k=\sin\phi_k$, we calculate  the columns $K_1, \dots, K_{10}\in\R^{16}\times \R^4$ of the differential of~$R$. We start with the columns with derivatives $D_{\circlearrowleft}$ on $\Sigma$.
\begin{equation*}
\begin{aligned}
K_1=  D_{\circlearrowleft x_1 y_1}R & =\big( (y_1s_1+ x_1 c_1, y_1 c_1- x_1 s_1,0,0), (0,0,0,0)\big)\in \R^{16} \times\R^4
  \\
K_2=   D_{\circlearrowleft x_1 x_2}R   & =\big( (x_2 s_1,x_2 c_1,-x_1 s_2 ,-x_1 c_2),(0,0,0,0)\big)
  \\
 K_3= D_{\circlearrowleft x_1 y_2}R & =\big(( y_2s_1, y_2 c_1, x_1c_2, -x_1s_2),(0,0,0,0)\big)
 \\
 K_4= D_{\circlearrowleft y_1 x_2}R & =\big((   -x_2  c_1  , x_2 s_1, -y_1 s_2, -y_1 c_2),(0,0,0,0)\big)
\\
 K_5= D_{\circlearrowleft y_1 y_2}R & =\big((-y_2c_1, y_2 s_1, y_1 c_2 , -y_1s_2),(0,0,0,0)\big)
\\
 K_6 = D_{\circlearrowleft x_2 y_2}R & =\big((0,0, y_2 s_2+x_2c_2,y_2c_2-x_2 s_2),(0,0,0,0)\big).
 \end{aligned}
\end{equation*}
The remaining four derivatives have the form
\begin{equation*}
\begin{aligned}
K_7= \p_{r_1 }R  & =((0,0,0,0) ,(1,0,0,0))
 \qquad  K_{8}=\p_{r_2 }R  & =((0,0,0,0) ,(0,1,0,0))
 \\
K_9=\p_{\phi_1 }R  & =((*,*,0,0) ,(0,0,1,0))
 \qquad
K_{10}=\p_{\phi_2 }R  & =((0,0,*,*) ,(0,0,0,1)).
\end{aligned}
\end{equation*}
The precise form of $*$ plays no role in the rank of the differential.

We first claim that $ K_1, \dots, K_6 \in\R^{16}\times\{0\}$ are independent.  Once the claim is proved, it will follows immediately from the form of $K_7, \dots, K_{10}$ that the rank of the differential is maximal.

To prove the claim,   note that equation  $\sum_{j=1}^6\la_j K_k=0 $ is equivalent to
\begin{equation*}
\begin{aligned}
 \la_1(y_1 s_1+x_1 c_1)+\la_2 x_2 s_1 +\la_3 y_{2} s_1-\la_4 x_2 c_1-\la_5 y_2 c_1\stackrel{E_1 }{=}0
\\
\la_1 (y_1 c_1-x_1 s_1)+\la_2 x_2 c_1+\la_3 y_2 c_1 +\la_4 x_2 s_1+\la_5 y_2 s_1
\stackrel{E_2 }{=}0
\\
-\la_2 x_1 s_2 +\la_3 x_1 c_2 -\la_4 y_1 s_2 +\la_5 y_1 c_2 +\la_6(y_2 s_2 +x_2 c_2)
\stackrel{E_3 }{=}0
\\ -\la_2 x_1 c_2 -\la_3 x_1 s_2 -\la_4 y_1 c_2 -\la_5 y_1 s_2 +\la_6 (y_2 c_2 -x_2 s_2)
\stackrel{E_4 }{=}0.
\end{aligned}
\end{equation*}
Recall first that $x_1, y_1, x_2, y_2$ are pairwise orthonormal and that $c_1^2+s_1^2=1$.
Projecting $E_1$ and $E_2$ along $x_1$ we see immediately that~$\la_1=0$. Projecting $E_5$ and $E_6$ along $x_2$ we get $\la_6=0$. Project then $E_1$ and $E_2$ along $x_2$. This gives $\la_2=\la_4=0$, because $c_1^2 + s_1^2=1$. For the same reason, projecting along $y_2$ $E_1$ and $E_2$ we discover that $\la_3=\la_5=0$.
\end{proof}

\begin{proof}[Proof of Theorem~\ref{magno}]
  We begin with the six derivatives along tangent directions to $\Sigma$. Then we will calculate the remaining four derivatives $\p_{r_j} $ and $\p_{\phi_j}$. We use the notation~\eqref{astuccio}.

Let us start by rotating the pair $u_1, v_1$. Note that this gives $(u_1\wedge v_1)'(0)=0$. It turns out that
the derivative $   D_{\circlearrowleft  u_1 v_1 }\Gamma$  gives
\begin{equation*}
\begin{aligned}
   D_{\circlearrowleft  u_1 v_1 }\Gamma &
   =\big( -r_1T_1 u_1, r_1r_2 Z_{12} v_1\wedge v_2+ r_1 r_2 Z_{21} u_1\wedge u_2\big)
   \\&= (-T_1 \a_1, Z_{21}\a_1\wedge \a_2+ Z_{12}\b_1\wedge \b_2)
\end{aligned}
\end{equation*}
(recall that $\a_j = r_j u_j$ and $\b_j = r_j v_j$).
Exchanging indices $1$ and $2$, we get
$   D_{\circlearrowleft u_2 v_2}\Gamma
     = (-T_2 \a_2,- Z_{12 }\a_1\wedge \a_2-Z_{21}\b_1\wedge \b_2)
$.
To get the third column of the differential of $\Gamma$ we rotate   $u_1$ and $v_2$. Here we have  $u_1'(0)= v_2$, $v_2'(0)=-u_1$ and $(u_1\wedge v_2)'(0)=0$. Thus
 \begin{equation*}
\begin{aligned}
 D_{\circlearrowleft  u_1 v_2 }\Gamma &
 =(-r_2 T_2 u_1, -r_1^2 U_1 v_1\wedge v_2 + r_2^2 U_2 u_1\wedge u_2)
 \\&
=\frac{1}{r_1 r_2}\big(-r_2^2 T_2\a_1, r_2^2 U_2\a_1\wedge \a_2 -r_1^2 U_1\b_1\wedge \b_2 \big).
\end{aligned}
\end{equation*}
 The fourth column can be obtained from the third exchanging $1$ and $2$:
 \begin{equation*}
  D_{\circlearrowleft  u_2 v_1 } =\frac{1}{r_1 r_2}  ( - r_1^2 T_1 \a_2,
  -r_1^2  U_1 \a_1\wedge \a_2 + r_2^2U_2\b_1\wedge \b_2).
 \end{equation*}
The fifth and the sixth columns take the form
\begin{equation*}
\begin{aligned}
 D_{{\circlearrowleft  u_1 u_2 }}\Gamma &= (0, r_1^2 U_1 u_2\wedge v_1+ r_1 r_2 Z_{12}u_2\wedge v_2 - r_1 r_2Z_{21}
 u_1\wedge v_1 -r_2^2 U_2 u_1\wedge  v_2  )
 \\&=\frac{1}{r_1 r_2}\big( 0, r_1^2 U_1 \a_2\wedge \b_1 + r_1^2Z_{12} \a_2\wedge \b_2
 -r_2^2 Z_{21} \a_1\wedge\b_1 - r_2^2 U_2\a_1\wedge \b_2\big).
 \end{aligned}
\end{equation*}
and
\begin{equation*}
D_{{\circlearrowleft  v_1 v_2 }}\Gamma =\frac{1}{r_1 r_2}\big(-r_2^2 T_2 \b_1+r_1^2 T_1 \b_2,
-r_2^2 Z_{12} \a_1\wedge \b_1 + r_1^2 U_1 \a_1 \wedge \b_2 + r_2^2 U_2 \b_1\wedge \a_2 + r_1^2 Z_{21}\a_2\wedge \b_2\big).
\end{equation*}
Derivatives with the variables $r_1$ and $r_2$ are
\begin{equation*}
\begin{aligned}
\p_{r_1}\Gamma & =
\big(T_1 v_1, 2r_1 U_1 u_1\wedge v_1+ r_2Z_{12} u_1\wedge v_2 +r_2Z_{21}u_2\wedge v_1
\big)
\\&=\frac
{1}{r_1}
  \big( T_1\b_1,2U_1 \a_1\wedge \b_1+Z_{12}\a_1\wedge \b_2 -Z_{21}\b_1\wedge \a_2\big),\quad\text{and}
  \\
  \p_{r_2}\Gamma &=\frac{1}{r_2}
\big( T_2 \b_2 ,  Z_{12}\a_1\wedge \b_2 -Z_{21} \b_1\wedge \a_2 + 2 U_2 \a_2\wedge \b_2\big).
\end{aligned}
\end{equation*}
 The last two columns can be obtained by differentiating along $\phi_1$ and $\phi_2$.
Using formulas for differentiating $T$ and $U$ from~\cite{MontanariMorbidelli17}, $T'(\phi)=-2V(\phi)$
and $U'(\phi)=\frac{\cos\phi}{\phi} V(\phi)$,  we get
\begin{equation*}
\begin{aligned}
 \p_{\phi_1}\Gamma &= \Big(-2V_1\b_1, \frac{\cos\phi_1}{\phi_1}V_1\a_1\wedge\b_1
 +(\p_1Z)(\phi_1,\phi_2)\a_1\wedge\b_2+(\p_2Z)(\phi_2, \phi_1)\a_2\wedge\b_1\Big)
 \\ \p_{\phi_2}\Gamma
 &= \Big(-2V_2\b_2,(\p_2Z)(\phi_1,\phi_2)\a_1\wedge\b_2+(\p_1Z)(\phi_2,\phi_1)\a_2\wedge\b_1
 +\frac{\cos\phi_2}{\phi_2}V_2\a_2\wedge \b_2\Big).
\end{aligned}
\end{equation*}
Collecting all computations of the ten columns and ignoring the positive terms $\frac{1}{r_1},\frac{1}{r_2}$ and $\frac{1}{r_1 r_2}$, we get the matrix in~\eqref{dieci}, as desired.

In order to prove the second part, it suffices to observe that the matrix in~\eqref{dieci} has the block form described by the following inclusions. Let $W_k$ be the $k$-th column. Then  $ W_1, W_2, W_3 , W_4\in  \Span \{\a_1, \a_2, \a_1\wedge\a_2,\b_1\wedge\b_2\}$, while    $W_5,W_6,\dots , W_{10} \in  \Span\{\b_1,\b_2,  \a_1\wedge\b_1,\a_1\wedge\b_2, \b_1\wedge\a_2,\a_2\wedge\b_2\}$.
\end{proof}

\section*{Declaration about Conflict of interest}
On behalf of all authors, the corresponding author states that there is no conflict of interest.

\section*{Acknowledgements}
   We thank the anonymous referee for the careful reading of the manuscript and for giving us several  useful suggestions on the presentation of the paper.

The authors are   supported by PRIN 2022 F4F2LH - CUP J53D23003760006 ``Regularity problems in sub-Riemannian structures''.

 The authors are also members of the {\it Gruppo Nazionale per
l'Analisi Matematica, la Probabilit\`a e le loro Applicazioni} (GNAMPA)
of the {\it Istituto Nazionale di Alta Matematica} (INdAM).

 \def\cprime{$'$} \def\cprime{$'$}
\providecommand{\bysame}{\leavevmode\hbox to3em{\hrulefill}\thinspace}
\providecommand{\MR}{\relax\ifhmode\unskip\space\fi MR }
\providecommand{\MRhref}[2]{%
  \href{http://www.ams.org/mathscinet-getitem?mr=#1}{#2}
}
\providecommand{\href}[2]{#2}


\end{document}